\documentclass{amsart}
\usepackage[x11names]{xcolor}
\usepackage{graphicx}
\usepackage{amssymb}
\usepackage{amsmath}
\usepackage{amsthm}
\usepackage{mathtools}
\usepackage{hyperref}

\theoremstyle{plain}
\newtheorem{theorem}{Theorem}[section]
\newtheorem{lemma}[theorem]{Lemma}
\newtheorem{proposition}[theorem]{Proposition}
\newtheorem{corollary}[theorem]{Corollary}
\newtheorem{conjecture}[theorem]{Conjecture}
\theoremstyle{definition}

\newtheorem{note}[theorem]{Note}

\makeatletter
\@namedef{subjclassname@2010}{%
  \textup{2010} Mathematics Subject Classification}
  \makeatother

\newcommand{\germ}{\operatorname{\ch@irxoperatorfont{germ}}}

\newcommand{\zetalpha}{\zeta^{(\alpha)}}
\newcommand{\zetak}{\zeta^{(k)}}
\newcommand{\etak}{\eta^{(k)}}

\newcommand{\etalpha}{\eta^{(\alpha)}}

\newcommand{\mc}{u}
\newcommand{\gameta}[1]{\widetilde\gamma_{#1}}

\newcommand{\diffint}[5]{{}^{#2}_{#5}\!#1^{(#4)}_{#3}}
\newcommand{\GL}[2]{\diffint{D}{GL}{#1}{#2}{}}    
    
\newcommand{\R}{\mathbb{R}} \newcommand{\Z}{\mathbb{Z}}
\newcommand{\C}{\mathbb{C}} \newcommand{\N}{\mathbb{N}}

\newcommand{\ourref}[2]{#1~\ref{#2}}

\title{Contributions to the theory of the Euler eta function}

\author[T. Caparatta]{Torre Caparatta}\email{tecapara@uncg.edu}
\author[S. Pauli]{Sebastian Pauli}\email{s\_pauli@uncg.edu}
\author[F. Saidak]{Filip Saidak}\email{f\_saidak@uncg.edu}
\address{Department of Mathematics and Statistics, University of North Carolina Greensboro, USA}
\date{\today}

\begin{document}

\maketitle

\begin{abstract}
We investigate the distribution of the zeros of the Euler's function $\eta(s)$ and its integral and fractional derivatives.
\end{abstract}

\section{Introduction}\label{sec eta intro}

We investigate the distribution of the zeros of the Euler eta function $\eta(s)$ (also called the alternating zeta function or the Dirichlet eta function) and its derivatives -- both integral and fractional -- in the complex plane. Our main goal will be to establish the following: 

\begin{itemize}
    \item zero-free regions for $\eta(s) - c$, where $c$ is a constant (in Section \ref{sec eta skorokhodov});
    \item zero-free regions for the derivatives $\eta^{(\alpha)}(s)$, where $\alpha > 0$ is real (in Section \ref{sec eta right});
    \item the non-vanishing of $\eta^{\prime}(s)$ in the left half of the critical strip, assuming the Riemann Hypothesis (in Section \ref{sec eta critical});
    \item intervals in the negative real numbers where $\eta'(s)$ has zeros (in Section \ref{sec eta left})
    \item estimates for the zero-counting functions $N_{\eta}^k(T)$, with $k \geq 1$, and in particular prove that $N_{\eta}^k(T) = N_{\zeta}(T) + O(\log T)$ (in Section \ref{sec eta number}).
\end{itemize}

Most of these results complement those found in the theory of the Riemann zeta function $\zeta(s)$ that are now considered classical. Finally, in Section \ref{sec eta conjecture} we discuss further observations
about the zeros of the derivatives of \(\eta(s)\), that lead to open problems and conjectures


For convenience, throughout the text we try to point out numerous parallels between fundamental properties of the zeros of $\zeta(s)$ and those of $\eta(s)$, but we also note a number of surprising differences, none stranger than the existence of a double zero (see Figure \ref{fig eta double zero}).

\subsection*{The Functions $\zeta(s)$ and $\eta(s)$}
Let us recall a few well-known historical facts. For $s=\sigma + it$,  the Riemann $\zeta$-function and the Euler $\eta$-function (its alternating version) are defined as: 
\begin{align}\label{eq eta}
  \zeta(s) := \sum_{n=1}^{\infty}\frac{1}{n^s} \quad \mathrm{and} \quad  \eta(s) := \sum_{n=1}^{\infty}\frac{(-1)^{n+1}}{n^s},
\end{align}
    where the expressions converge for all $s$ with $\sigma >1$ and $\sigma >0$, respectively. Therefore, unlike the $\zeta$-function, the Euler $\eta$-function does not have a pole at $s=1$. In fact, integrating the geometric series $1 - x + x^2 - x^3 +  \cdots = (1+x)^{-1}$ (with $|x| < 1$), one finds (as already discovered by Mengoli in 1650 \cite{mengoli}):  
\[
\eta(1) = \lim_{x \to 1^-} \left[ x - \frac{x^2}{2} + \frac{x^3}{3} - \frac{x^4}{4} + \cdots \right]  =  \int_0^1 \frac{1}{1+x}  \; dx = \log 2.
\]
Also, the lack of a pole of \(\eta(s)\) makes it an entire function that can be expressed in the following infinite series form:
\begin{equation}\label{eq eta laurent}
\eta(s)=\sum_{j=0}^\infty (-1)^j\frac{\gameta{j}}{j!}(s-1)^j,
\end{equation}
where $\gameta{j}$ are constants. Moreover, if we let $\Gamma(s) := \int_0^{\infty} t^{s-1} e^{-t} \; dt$ be the Legendre \cite{le-1809} integral for the $\Gamma$-function (first defined by Euler \cite{euler1730} in a 1730 letter to Goldbach), then the Euler $\eta$-function can be represented in the form (analogous to Abel's formula \cite{abel1823} involving $\zeta(s)$ and $\Gamma(s)$):
\begin{align*} 
    \eta(s)\Gamma(s) = \int_0^\infty\frac{t^{s-1}}{e^t+1}dt.
\end{align*}
Furthermore, one can easily observe that
\begin{align}\label{altern zeta}
    \eta(s) = \zeta(s) - 2 \frac{\zeta(s)}{2^s} = \left(1-2^{1-s}\right)\zeta(s),
\end{align}
and this implies that all zeros of $\zeta(s)$ are also zeros of $\eta(s)$, while additional zeros come from the solutions of the equation $1 = 2^{1-s}$, namely: $s_n=1+\frac{2n\pi i}{\log 2}$, where $n$ is a nonzero integer. 
Also, as Euler noted in 1749 \cite{Euler1749}: 
$$ 
\frac{\eta(1-s)}{\eta(s)} = - \frac{\Gamma(s)(2^s - 1)}{(2^{s-1} - 1) \pi^s} \cos \left( \frac{\pi s}{2} \right),
$$
which is equivalent to 
\begin{align}\label{eta functional}
    \eta(-s) = 2\pi^{-s-1}s\Gamma(s)\sin\left(\frac{\pi s}{2}\right)\frac{1-2^{-s-1}}{1-2^{-s}}\eta(s+1),
\end{align}
the so-called functional equation of the Euler \(\eta\)-function (of which Hardy  gave a simple proof in 1922 \cite{hardy}).
In more recent times there has been a renewed interest in $\eta(s)$: various new properties were studied by Sondow \cite{Son03}, Milgram
\cite{Mil13}, Alzer and Kwong \cite{AK15}, Boyadzhiev and Frontczak \cite{BF21}, and others.

\subsection*{Derivatives}
We now consider the $k$-th derivatives of $\eta(s)$, starting with the observation that, for $k\in \N$ and $\Re(s)>0$ the definition \eqref{eq eta} readily implies
\begin{align*}
    \eta^{(k)}(s)=(-1)^k\sum_{n=1}^{\infty}\frac{(-1)^{n}\log^k n}{n^s},
\end{align*}
while taking a derivative of both sides of \eqref{altern zeta}, gives 
\begin{equation*}
\eta^{\prime}(s) = 2^{1 -s} \log 2 \zeta(s) + (1 - 2^{1-s}) \zeta^{\prime}(s), 
\end{equation*}
useful, among other things, for computing special values, like $\eta^{\prime}(0) = \frac{1}{2} \log \left(\frac{\pi}{2} \right)$
and $\eta^{\prime}(1) = \gamma \log 2 + \frac{1}{2}(\log 2)^2$.

In this paper we investigate the locations of zeros of $\eta(s)$ and its derivatives $\eta^{(\alpha)}(s)$, 
for real  $\alpha> 0$.   When it comes to the fractional derivatives, one has two main options, namely the Gr\"unwald-Letnikov fractional derivative \cite{g:1867,l:1869-1,l:1869-2}
and the Riemann-Liouville fractional derivative.  Since many results for integral derivatives easily generalize to the (reverse) Gr\"unwald-Letnikov fractional derivative (for the Riemann zeta function case see, for example, \cite{fps} motivated by \cite{kreminski03}) we consider this fractional derivative and formulate our results for both the integral and this fractional derivative when possible.
For all real $\alpha > 0$ and $c\in\C$ we have $\GL{\leftarrow}{\alpha}[c] = 0$ and for $m\ne 0$ we have $\GL{\leftarrow}{\alpha}\left[e^{ms}\right]  = m^\alpha e^{ms}$.  Thus
\begin{align*}
    \etalpha(s) := \GL{\leftarrow}{\alpha}[\eta(s)] = 
    (-1)^\alpha\sum_{n=1}^{\infty}\frac{(-1)^{n}\log^{\alpha} n}{n^s}, 
\end{align*}
where $s\in\C$ with $\Re(s)>0$.  Repeated differentiation of \eqref{eq eta laurent} yields
\begin{equation}\label{eq eta laurent diff}
\eta^{(k)}(s)=(-1)^k\sum_{j=0}^\infty (-1)^{j}\frac{\gameta{k+j}}{j!}(s-1)^j,
\end{equation}
which implies that $\gameta{k}=\eta^{(k)}(1)$.
The most natural way to generalize \eqref{eq eta laurent diff} to the $\alpha$-th fractional derivatives is by defining \(\widetilde\gamma_\alpha\) for \(\alpha\ge 0\) by:  
\begin{equation}\label{eta power}
\eta^{(\alpha)}(s)=(-1)^{\alpha}\sum_{j=0}^\infty (-1)^{j}\frac{\gameta{\alpha+j}}{j!}(s-1)^j,
\end{equation}
such that \(\gameta{\alpha}=\eta^{(\alpha)}(1)\) for $\alpha\in[0,\infty)$. 
This immediately yields 
\begin{equation}\label{eq gameta}
\gameta{\alpha}=\sum_{n=1}^{\infty}(-1)^{n}\frac{\log^\alpha(n)}{n}.
\end{equation}
\begin{proposition}
For all non-negative integers \(k\),
\(b\in\Z\), and \(\sigma\in\R\) we have \(\Re\left(\eta^{(k+{\frac{1}{2}})}(\sigma)\right) =  0\).
\begin{proof}
    Since $\gameta{\frac{1}{2}+k} \in \R$ by (\ref{eq gameta}), we deduce from  (\ref{eta power}) that
    \[
        \Re\bigl(\eta^{(k+\frac{1}{2})}(\sigma)\bigr)=\Re\Bigl((-1)^{\frac{1}{2}+k}\sum_{j=1}^{\infty}(-1)^j\frac{\gameta{j+k+\frac{1}{2}}}{j!}(\sigma-1)^j\Bigr)=0. \qedhere
    \]
\end{proof}
\end{proposition}

\section{Non-Vanishing of $\eta(s) - c$}\label{sec eta skorokhodov}

We start by looking at the properties of the functions $\eta(s) - c$, where $c \in \mathbb{R}$ is a fixed constant. In particular, we study the non-vanishing of $\eta(s) - c$, which 
(while interesting in its own right) helps describe the paths \(s:[0,1]\to\C\) given by \(\eta(s(c))-c=0\) from a zero \(s(0)\) of \(\eta\) to a zero \(s(1)\) of \(\eta-1\).  
These paths link the zeros of \(\eta\) with the zeros of its fractional derivatives see Figures \ref{eta plot} and \ref{eta alpha}.

\begin{figure}[ht]
    \centering
    \includegraphics[width = \textwidth]{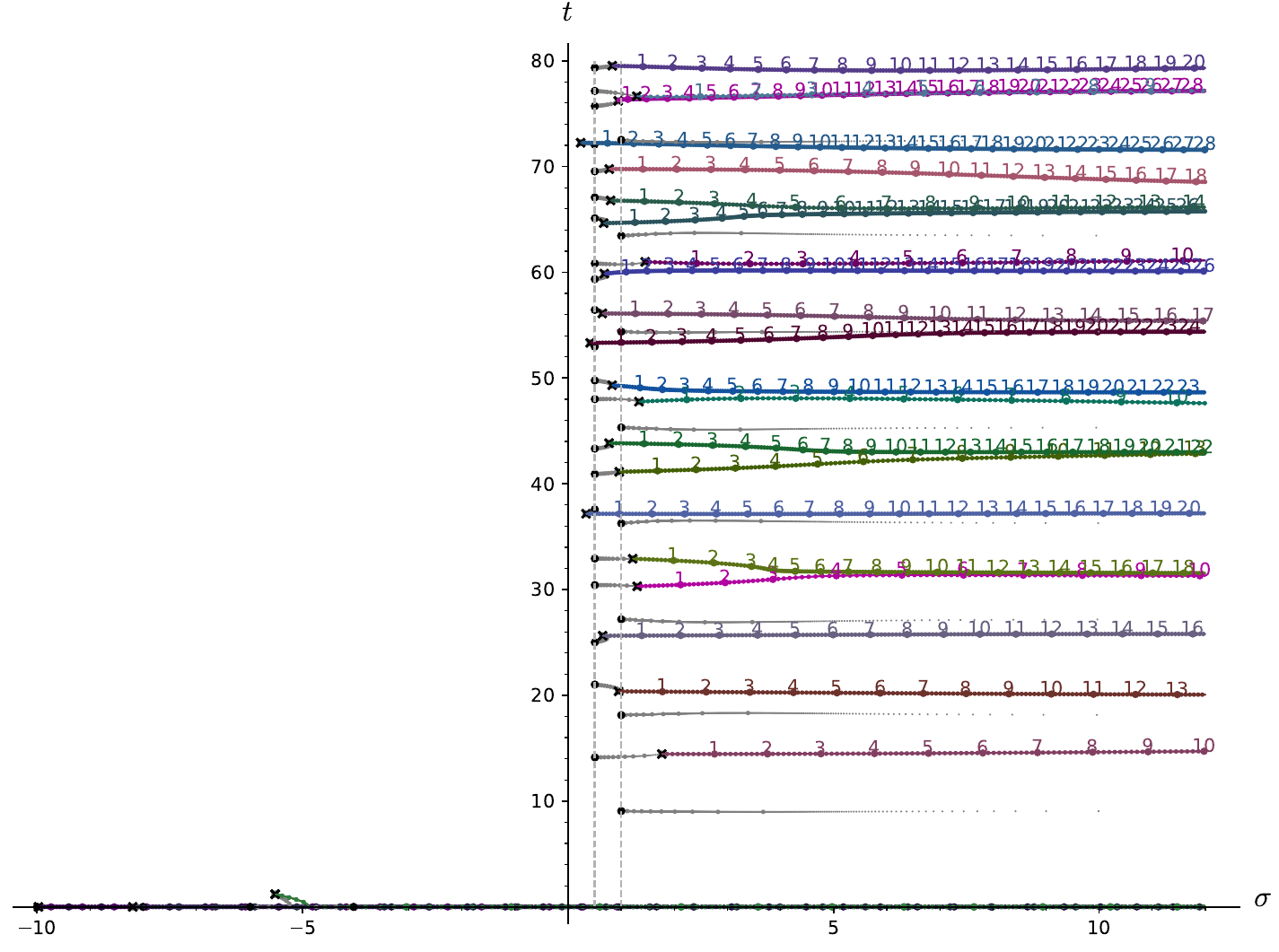}
    \caption{Zeros of the fractional derivatives of the Euler \(\eta\) function with zeros of $\eta(s)$ represented by $\color{black}\bullet$,  zero of $\eta(s)-1$ represented by \textsf{x}, and zeros of $\eta^{(k)}(s)$r repersented by $\color{blue}\bullet^{k}$. Grey lines are zeros of \(\eta(s)-c\) for \(0\le c\le 1\).   For the non-real zeros on the left half plane also see Figure \ref{fig eta double zero}.}
    \label{eta plot}
\end{figure} 

First, we prove the following special case: 

\begin{proposition}\label{sch eta}
The function $\eta(s)$ is distinct from unity at $\sigma\in(\sigma_0,\infty)$, where
\[\sigma_0=1.940101683745\dots\]
is the zero of the function \(f(\sigma) = 1+2^{-\sigma}-(1-2^{-\sigma})\zeta(\sigma)\) with \(\sigma>1\).
\end{proposition}

\begin{proof}
    Using a similar relation as in \cite[p.1293]{s:3} for \(\sigma>1\) we have
    \begin{align}\label{eta odd1}
        (1+2^{-s})\eta(s) = \sum_{n=0}^{\infty}\frac{1}{(2n+1)^s}.
    \end{align}
    One obtains 
    \begin{align*}
        (1+2^{-s})(\eta(s)-1) = -2^{-s}+\sum_{n=1}^{\infty}\frac{1}{(2n+1)^s}.
    \end{align*}
    With \(|a|+|b|\geq |a-b|\geq |a|-|b|\) we get 
    \begin{align}\label{eta odd3}
    |1 + 2^{-s}||\eta(s)-1|\geq | -2^{-s}|-\sum_{n=1}^{\infty}\frac{1}{|(2n+1)^s|}=1+2^{-\sigma} -\sum_{n=0}^{\infty}\frac{1}{(2n+1)^{\sigma}}.
    \end{align}
    Therefore, using (\ref{eta odd1}), we can rewrite (\ref{eta odd3}) as
    \[
        |1+2^{-s}||\eta(s)-1|\geq 1+2^{-\sigma}-(1-2^{-\sigma})|\eta(s)|
    \geq1+2^{-\sigma}-(1-2^{-\sigma})\zeta(\sigma)=f(\sigma)\qedhere
    \]
\end{proof}

The simplest forms of the zero-free regions for $\eta(s)-c$ can be obtained with little more than basic information concerning the modulus of the functions involved:

\begin{lemma}
    If $c\in[0,1)$ and $|\sin(t\log 2)|\geq 2^\sigma \zeta(\sigma)-2^\sigma - 1$, then 
    $$\eta(\sigma +it)-c \ne 0. $$
\end{lemma}

\begin{proof}
    We consider the imaginary part of $\eta(s)-c$ and obtain
    \begin{align*}
        |\Im(\eta(s)-c)|&\geq \left|\frac{\sin(t\log 2)}{2^\sigma} \right|-
        \left|\sum_{n=3}^\infty \frac{(-1)^{n+1}}{n^s}\right|\\
        &\geq \left|\frac{\sin(t\log 2)}{2^\sigma} \right|-\zeta(\sigma)+1+\frac{1}{2^\sigma} 
             > 0
    \end{align*}
   whenever $|\sin(t\log 2)|\geq 2^\sigma \zeta(\sigma) -2^\sigma-1$, as wanted.
\end{proof}
Similarly, we have:
\begin{lemma}
    If $c\in[0,1)$ and $-\cos(t \log 2 ) \geq 2^\sigma\zeta(s) -2^\sigma-1$, then 
    $$\eta(\sigma +it)-c \ne 0. $$
\end{lemma}
\begin{proof}
    For the real part of $\eta(s)-c$ we obtain
    \begin{align*}
        \Re(\eta(s)-c) 
        &= 1-c-\frac{1}{2^\sigma}\cos(t \log 2) +\frac{1}{3^\sigma}\cos(t \log 3)-\cdots+\cdots\\
        &\geq -\frac{1}{2^\sigma}\cos(t\log2) -\Re\left(\eta(s) -1 + \frac{1}{2^s
        }\right) \\
        &\geq -\frac{1}{2^\sigma}\cos(t \log 2) -\left(\zeta(\sigma)-1-\frac{1}{2^\sigma}\right) > 0,
    \end{align*}
 whenever $ -\cos(t \log 2) \geq 2^\sigma\zeta(s) -2^\sigma-1$. 
\end{proof}

\begin{note}\label{note -c}
In Figure \ref{eta plot}  we notice that the real part
of the paths of zeros of \(\eta(s)-c\) with starting point \(s_n=1+\frac{2n\pi i}{\log 2}\) approaches 
infinity as \(c\) approaches 1, while the path that start at a zero of \(\zeta(s)\) lead to a
zero of \(\eta(s)-1\).  

In the case of the Riemann zeta function there  are paths  of zeros of \(\zeta(s)-c\) that start at a zero 
of \(\zeta\) whose real part approaches infinity as \(c\) approaches 1 (compare \cite{s:3} and \cite{bp}).
\end{note}

\begin{figure}[ht]
    \centering
    \includegraphics[width=.8\textwidth]{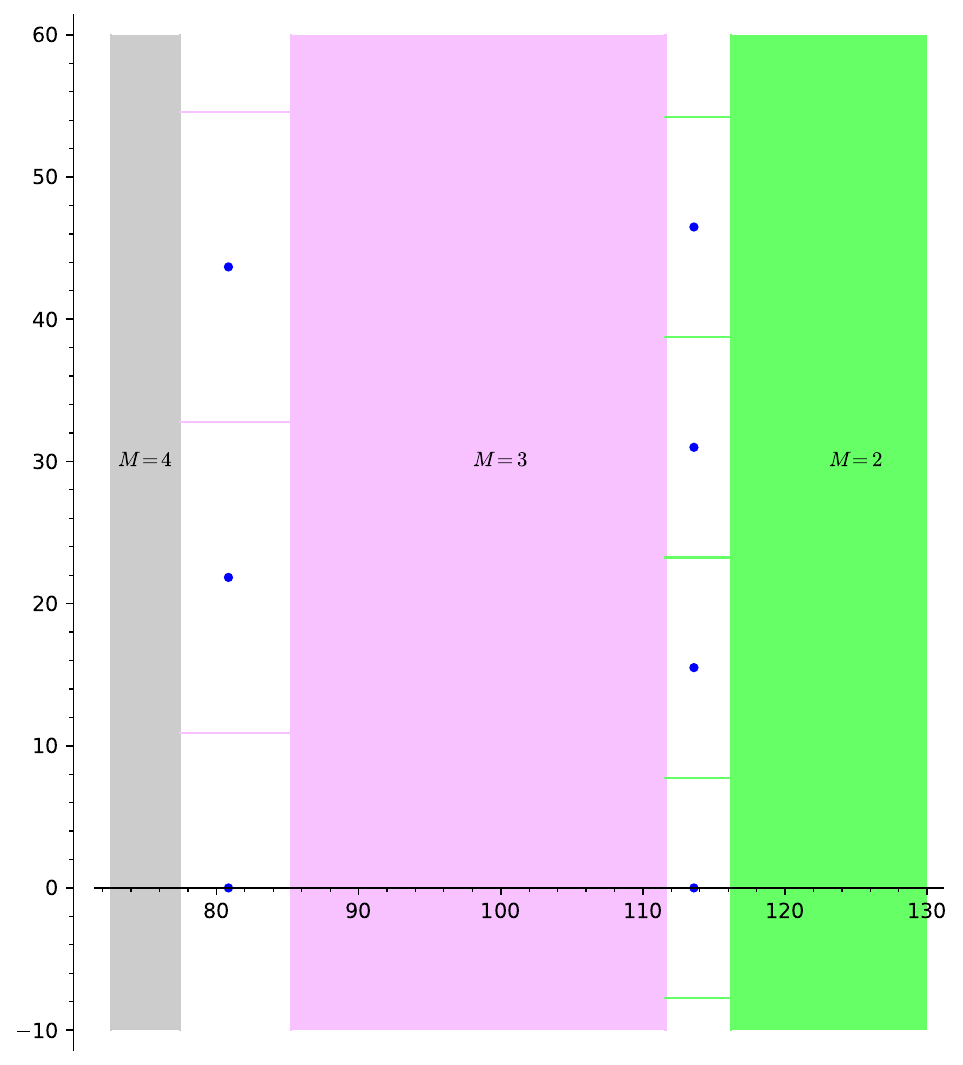}
    \caption{Zero-free regions and lines for $\eta^{(100)}(s)$ with zeros  of $\eta^{(100)}$ represented by \(\color{blue}\bullet\).}
    \label{eta cloud}
\end{figure}

\section{Right Half-Plane}\label{sec eta right}

Now we turn our attention to the derivatives of $\eta(s)$, as defined in Section 1.2. We start by focusing on zero-free regions of $\etalpha(s)$ (where $\alpha > 0 $ is real) in the right half-plane. The situation is similar to the zero-free regions of $\zetalpha(s)$ (which we have studied in \cite{bps} and \cite{fps}), so we'll find it useful to re-introduce  some of the relevant notation. 

\subsection*{Statements of Two Theorems}

Set $Q^{\alpha}_n(s):=(\log n)^\alpha/n^s$ and write the Dirichlet series of $(-1)^{\alpha} \etalpha(s)$ as
\begin{equation}\label{eq Q}
(-1)^{\alpha} \etalpha(s)=(-1)^{\alpha}\sum_{n=2}^\infty \frac{(-1)^{n-1}\log^\alpha n}{n^s}=(-1)^{\alpha}\sum_{n=2}^\infty (-1)^{n-1}Q^\alpha_n(s).
\end{equation}
 In what follows, we prove the existence of zero-free regions where one of the terms of  \eqref{eq Q}, say $Q^{\alpha}_M(\sigma)$, 
dominates the rest of the series, that is, when
\begin{equation*}
Q^\alpha_M(\sigma) > \sum_{n\ne M} Q^\alpha_n(\sigma),
\end{equation*}
and, in a complementary fashion, we look for the zeros of $ \etalpha(s)$ near the regions of the complex plane where 
$Q^{\alpha}_M(s) = Q^{\alpha}_{M+1}(s)$, in other words where no term of the series can attain dominance. This happens 
when the cancellation of those two terms occurs, at
\begin{equation*}
q_M:=\frac{\log \left( \frac{\log M}{\log (M+1)}  \right)}{\log \left( \frac{M}{M+1} \right)}. 
\end{equation*}

Our first main goal will be to prove: 

\begin{theorem}\label{thmone}

Let $\alpha>0$. The following statements are true: 
\begin{itemize}
\item[(a)] For all $\sigma> q_2\alpha +2.6$, we have $\etalpha(s)\ne 0$.
\item[(b)] If $q_3 \alpha + 4\log3 < q_2 \alpha - 2$, then $\etalpha(s)\ne 0$ for
\[
q_3 \alpha + 4\log3 \le\sigma\le q_2 \alpha  - 2.
\]
\item[(c)] If $M\in\N$, $M>3$, and 
$q_M \alpha+(M+1)\mc\le q_{M-1} \alpha -M\mc$, then $\etalpha(s)\ne 0$ in the regions 
\[
q_M \alpha+(M+1)\mc\le\sigma\le q_{M-1} \alpha -M\mc, 
\] 
where $\mc\in(0,\infty)$ is a solution of $1-\frac{1}{e^\mc-1}-\frac{1}{e^{\mc}}(1+\frac{1}{\mc})\ge0$.  
\end{itemize}
\end{theorem}
The value of $\mc\in(0,\infty)$ that gives us the widest zero-free regions is $\mc=1.1879426249\dots$, which is the solution of the equation 
\[
1-\frac{1}{e^\mc-1}-\frac{1}{e^{\mc}}\left(1+\frac{1}{\mc}\right)=0.
\]

The main argument of our proof runs along lines parallel to those we have used in \cite{fps2}, see Figure \ref{e-frac-rouche}.  
We know that $|\etalpha(s)| \geq |\zetalpha(s)|$. Let $S_M^\alpha$ be the vertical strip between the zero-free regions obtained from
the dominance of $Q_M^{\alpha}(q_M \alpha)$ and $Q_{M+1}^{\alpha}(q_M \alpha)$ in (\ref{eq Q}), respectively, 
as described in Theorem \ref{thmone}. The strip $S_M^\alpha$ exists when $\alpha$ reaches 
\[
A_M:=\left\{
\begin{array}{ll}
\frac{4\log3+2}{q_2-q_3} & \mbox{ if }M=2\\[1ex]
\frac{(2M+3)\mc}{q_M-q_{M+1}} &\mbox{ if }M>2.
\end{array}\right.
\]

\begin{figure}
    \centering
    \includegraphics[width=.9\textwidth]{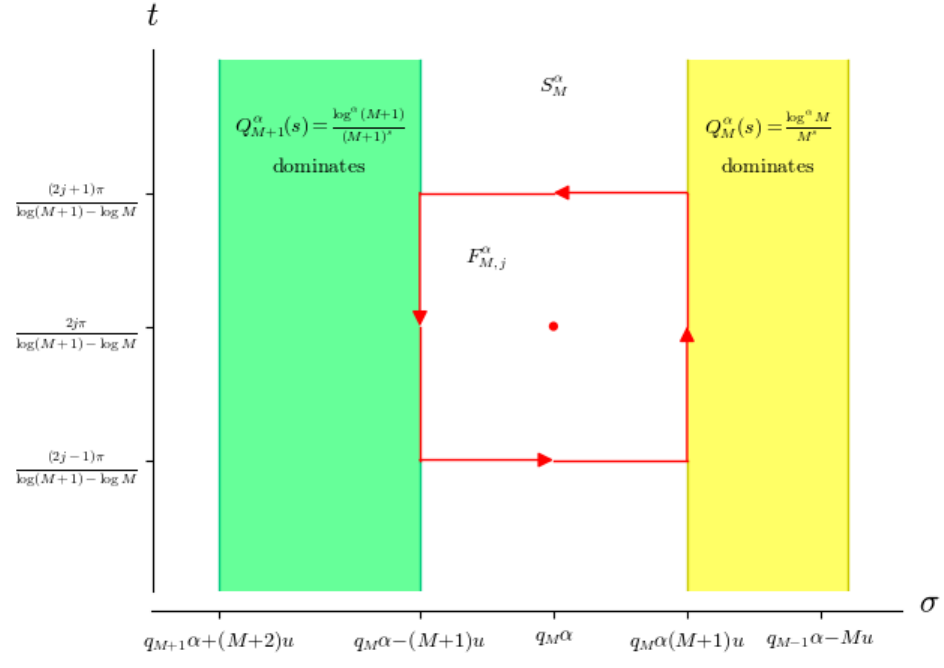}
    \caption{Regions $F_{M,j}^{\alpha}$ containing exactly one zero of $\etalpha(\sigma+it)$. Rouch{\'e}'s theorem is used to prove its simplicity using the zero of $Q^{\alpha}_{M}(s)-Q^{\alpha}_{M+1}(s)$ represented by ${\color{red}\bullet}$.}
    \label{e-frac-rouche}
\end{figure}

Recall that $Q_M^{\alpha}(q_M \alpha)=Q_{M+1}^{\alpha}(q_M \alpha)$.
Considering the imaginary parts of the solutions of $Q_M^{\alpha}(q_M \alpha+it)-Q_{M+1}^{\alpha}(q_M \alpha+it)=0$ we find
that $\etalpha(\sigma+it)\ne 0$ for $\sigma\in S_M^\alpha$ and
\begin{equation}\label{eqconjt}
t = \frac{\pi (2J+1)}{\log(M+1)-\log(M)}
\end{equation} 
for $J\in\Z$.
Together with the border of the zero-free regions to the left and right of $S_M^\alpha$ the lines from (\ref{eqconjt}), 
for $J=j$ and $J=j-1$, where $j\in\Z$ form a contour around the zero 
\begin{equation*}
q_M\cdot\alpha+\frac{2\pi j}{\log(M+1)-\log(M)}\,i
\end{equation*}
of $Q_M^{\alpha}(q_M \alpha+it)-Q_{M+1}^{\alpha}(q_M \alpha+it)$.
Just as in \cite{bps}, Rouch\'e's theorem can be also used to show that there is exactly one
zero of the fractional derivatives $\etalpha(s)$ in the rectangular area shown in Figure \ref{e-frac-rouche}. In other words, a natural generalization of \cite[Theorem 2.2]{bps} can be obtained, {\em mutatis mutandis}, 
by replacing integer values of $k$ by positive real numbers $\alpha$:

\begin{theorem}\label{thmboxzeroeta} 
Let $M\ge2$ denote a natural number, $j\in\Z$, and $\alpha>A_M$.
Let $F^\alpha_{M,j}\subset S_M^\alpha$ be given by 
\(
\frac{\pi (2j-1)}{\log(M+1)-\log(M)}<t<\frac{(2j+1)\pi }{\log(M+1)-\log(M)}.
\)
Then $F^\alpha_{M,j}$ contains exactly one zero of $\etalpha(s)$, and the zero is simple.
\end{theorem}

Computations suggest that the zeros in the regions $F_{M,j}^\alpha$ 
form continuous, mostly horizontal path.
We observe that the paths of zeros of fractional derivatives passing through the regions 
$F_{M,j}^\alpha$ (with $j>0$) on the left end at zeros of $\eta(s)-1$, where $-\frac{1}{2}<\Re(s)<1.9402$, discussed in Theorem \ref{sch eta}.

Far enough to the right the existence of these path of zeros is a direct consequence of Theorem \ref{thmboxzeroeta}:
Let $M\in\Z$, $M\ge 2$ and $\alpha>A_M$ so that $S_M^\alpha$ is non empty.
Then for each $j\in\Z$ there is $s=\sigma+it\in F^\alpha_{M,j}$ such that $\etalpha(s)=0$.
Because $s$ is a simple zero of $\etalpha(s)$ we have that $\eta^{(\alpha+1)}(s)\ne 0$.
By the implicit function theorem there is an analytic function $z$ defined 
on an open neighborhood $U\subset\C$ of $\alpha$ such that $\eta^{(\beta)}(z(\beta))=0$ for $\beta\in U$.
As this holds for all $\alpha>A_M$ we obtain a
function $z$ that is analytic on an open neighborhood of 
$(A_M,\infty)$ in $\C$
and thus analytic on $(A_M,\infty)$.

\begin{corollary}\label{cor curve eta}
Let $M\in\N$ with $M\ge 2$ and $j\in\Z$.
The zeros $s=\sigma+it$ 
of $\etalpha(s)$ 
for $\alpha>A_M$ 
with 
\[ 
\frac{\pi (2j+1)}{\log(M+1)-\log(M)}<t<\frac{\pi2j}{\log(M+1)-\log(M)}
\]
are images of an analytic function
\(
z:(A_M,\infty)\to\C.
\)
\end{corollary}

\subsection*{Proofs of Two Theorems}

Before we give a proof of Theorem 3.1, we need one auxiliary lemma:

\begin{lemma}\label{lemline eta}
Let $M \geq 2$ and $\alpha \in \mathbb{R}$. If $s \in S_M^\alpha$ (defined above), then $\etalpha(s)\ne 0$ for
\[
s=\sigma+i\cdot\frac{\pi (2j+1)}{\log(M+1)-\log M}.
\]
\end{lemma}

\begin{proof}
In the center of the strip $S_M^\alpha$, that is on the line $\sigma=q_M\alpha$ we have $|Q_M^\alpha(s)|=|Q_{M+1}^\alpha(s)|$.
We consider the line segments in $S_M^\alpha$ with
\[
q_M \alpha-(M+1)\mc\le \sigma \le q_M \alpha+(M+1)\mc.
\]
and
\[
t=\frac{\pi (2j+1)}{\log(M+1)-\log M}, \mbox{ where }j\in\Z,
\]
see \ourref{Figure}{e-frac-rouche}. Everywhere hereafter we write $H_M^{\alpha}(s)$ for the "head" and $T_M^{\alpha}(s)$ for the "tail" of the series $\etalpha(s)$ split by $Q_M^{\alpha}(s)$:
\[H_M^{\alpha}(s) := \sum_{n=2}^{M-1}|Q_n^{\alpha}(s)|
\text{ and }
    T_M^{\alpha}(s):= \sum_{n=M+1}^{\infty}|Q_n^{\alpha}(s)|. 
\]


Our choice of $t$ gives $Q_M^\alpha(q_M \alpha+it)-Q_{M+1}^\alpha(q_M \alpha+it)=0$ (cf. (\ref{eqconjt})) 
and therefore $\cos(t\log M)=-\cos(t\log(M+1))$ and $\sin(t\log M)=\sin(t\log(M+1))$.
We set $s=\sigma+it$, with $t$ and $\sigma$ as above, and consider the real and imaginary parts of the expression
\[
\etalpha(s)=\sum_{n=2}^\infty (-1)^{n-1}\left(\cos(t\log n)-i\cdot\sin(t\log n)\right)Q_n^\alpha(\sigma).
\]
With $|\Im(Q_{n}^\alpha(s)|\le Q_n^\alpha(\sigma)$ and $|\Re(Q_{n}^\alpha(s)|\le Q_n^\alpha(\sigma)$ we obtain
\begin{align*}
|\Re(\etalpha(s))|\ge& |\cos(t\log M)Q_M^\alpha(\sigma)+\cos(t\log(M+1))Q_{M+1}^\alpha(\sigma)|\\&-H_M^\alpha(\sigma)-T_{M+1}^\alpha(\sigma),\\
|\Im(\etalpha(s))|\ge& |\sin(t\log M)Q_M^\alpha(\sigma)+\sin(t\log(M+1))Q_{M+1}^\alpha(\sigma)|\\&-H_M^\alpha(\sigma)-T_{M+1}^\alpha(\sigma).
\end{align*}

If $t = 0$, the situation is trivial.  
If $t\ne 0$, then we either have $|\sin(t\log M)|\ge\sin(\pi/4)=1/\sqrt{2}$ or $|\cos(t\log M)|\ge\cos(\pi/4)=1/\sqrt{2}$.
Because $|\etalpha(s)|\ge |\Re(\etalpha(s))|$ and $|\etalpha(s)|\ge |\Im(\etalpha(s))|$ we get:
\begin{eqnarray*}
|\etalpha(s)|
&\!\!\!\ge\!\!\!& \frac{1}{\sqrt{2}}\left(Q_M^\alpha(\sigma)+Q_{M+1}^\alpha(\sigma)\right)-H_M^\alpha(\sigma)-T_{M+1}^\alpha(\sigma)\\
&\!\!\!=\!\!\!&   Q_M^\alpha(\sigma)\left(\frac{1}{\sqrt{2}}+\frac{1}{\sqrt{2}}\frac{Q_{M+1}^\alpha}{Q_M^\alpha}(\sigma)
-\frac{H_M^\alpha}{Q_M^\alpha}(\sigma)-\frac{Q_{M+2}^\alpha}{Q_M^\alpha}(\sigma)-\frac{T_{M+2}^\alpha}{Q_M^\alpha}(\sigma)\right)\\
&\!\!\!=\!\!\!&   Q_M^\alpha(\sigma)\!\left(\!\frac{1}{\sqrt{2}}-\frac{H_M^\alpha}{Q_M^\alpha}(\sigma)
+\frac{Q_{M+1}^\alpha}{Q_M^\alpha}(\sigma)\!
\left(\!\frac{1}{\sqrt{2}}-\frac{Q_{M+2}^\alpha}{Q_{M+1}^\alpha}(\sigma)-\frac{T_{M+2}^\alpha}{Q_{M+1}^\alpha}(\sigma)\!\right)\!\!\right)
\end{eqnarray*}

From the proof of \ourref{Theorem}{thmone}~(b) we know that for $\sigma\ge q_{M+1}\alpha+(M+2)\mc$
and $\mc=1.1879426249\dots$ 
\begin{align*}
\frac{1}{\sqrt{2}}-\frac{Q_{M+2}^\alpha}{Q_{M+1}^\alpha}(\sigma)-\frac{T_{M+2}^\alpha}{Q_{M+1}^\alpha}(\sigma)
&\ge\frac{1}{\sqrt{2}} -\frac{Q_{M+2}^\alpha}{Q_{M+1}^\alpha}(\sigma)\left(1+R_{M+2}(\sigma)\right)\\
&\ge\frac{1}{\sqrt{2}} - \frac{1}{e^\mc}\left(1+\frac{1}{\mc}\right) > 0.
\end{align*}
Similarly, since $\frac{H_M^\alpha}{Q_M^\alpha}(\sigma)$ is increasing in $\sigma$, compare \cite[Proof of Theorem 1 (b)]{zero-free-fract}
and because $\sigma<q_{M-1}\alpha-M\mc$, we get that
\[
\frac{1}{\sqrt{2}}-\frac{H_M^\alpha}{Q_M^\alpha}(\sigma)\ge\frac{1}{\sqrt{2}}-\frac{H_M^\alpha}{Q_M^\alpha}(q_{M-1}\alpha-M\mc)\ge\frac{1}{\sqrt{2}}-\frac{1}{e^\mc-1}>0,
\]
which concludes the proof of the lemma.
\end{proof}

\begin{proof}[Proof of \ourref{Theorem}{thmboxzeroeta}]
Let $Z(s)=Q_M^\alpha(s)-Q_{M+1}^\alpha(s)$.
It is easy to check that the function $Z(s)$ has exactly one (simple) zero in $R_j$, namely
\[
s=q_M \alpha+ i\cdot\frac{2\pi j}{\log(M+1)-\log M}.
\]
In order to be able to apply Rouch\'e's Theorem we need to show that $|\etalpha(s)-Z(s)|<|Z(s)|$ for all $s$ on $R_j$.
The vertical sides of $R_j$ are in the zero-free regions for $M$ and $M+1$.
As shown in the proof of \ourref{Theorem}{thmone} 
the term $Q_M^\alpha(s)$ dominates $\etalpha(s)$
on the right vertical side of $R_j$
and 
the term $Q_{M+1}^\alpha(s)$ dominates $\etalpha(s)$
on the left vertical side of $R_j$.  Thus $|\etalpha(s)-Z(s)|<|Z(s)|$ on the vertical sides of $R_j$.
Furthermore we have seen in the proof of \ourref{Lemma}{lemline eta} that
$Z(s)=Q_M^\alpha(s)+Q_{M+1}^\alpha(s)$ dominates $\etalpha(s)$ on the horizontal sides of $R_j$.
Hence $|\etalpha(s)-Z(s)|<|Z(s)|$ on the horizontal sides of $R_j$.
\end{proof}

\begin{note}
We know that the fractional derivatives $\zeta^{(\alpha)}(s)$ do not have zeros on the positive real axis \cite[Proposition 1]{fps-icms}. However, the same cannot be said about the Euler $\eta$-function. From the alternating property, the imaginary part of the dominating terms used in Theorem \ref{thmboxzeroeta} shifts the boxes to include the real axis. This guarantees existence of zeros of $\etalpha(s)$ on the real axis. Because $\eta^{(k)}(s)$ has the same number of zeros as $\zeta(s)$ (see Section \ref{sec eta number}), we expect that these zeros on the real axis lie on paths of fractional derivatives that originate on the left half-plane, see Figure \ref{eta alpha}. 
\end{note}

\begin{figure}
   \centering
    \includegraphics[width=.9\textwidth]{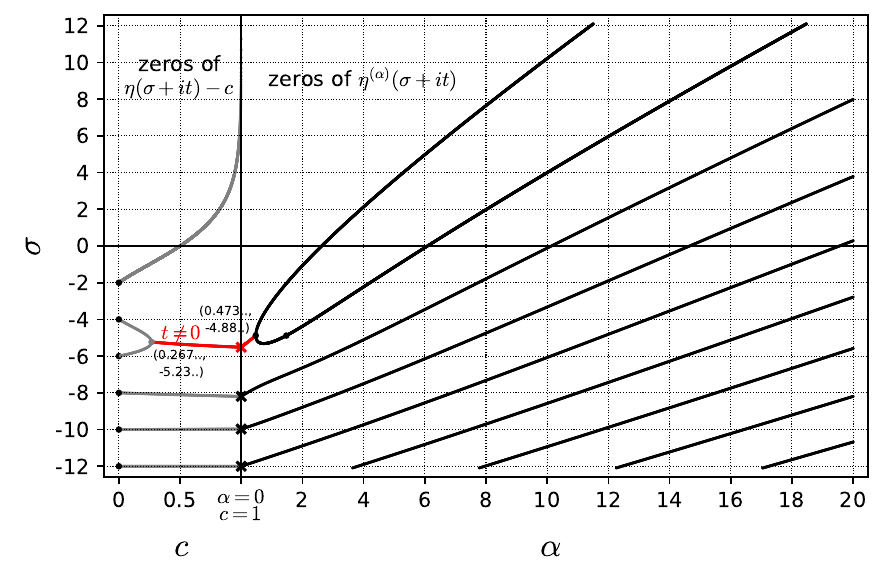}
    \caption{Paths \(s:(0,20]\to\C\) given by $\etalpha(s(\alpha))=0$ originating on the left half plane        in the $\sigma-\alpha$ plane on the right and the path \(s:(0,1]\to\C\) given by $\eta(s(c))-c=0$ in the $c-\alpha$ plane.  We denote zeros of \(\eta\) by \(\bullet\) and by {\color{gray}\textsf{x}} the zeros of $\eta-1$.  Only the path from
    \((c\approx 0.267\), \(\sigma-5.23\dots)\) to the zeros of \(\eta-1\) near \(-5.516 
    \pm 1.201 i\) 
    and from there to \(\alpha\approx 0.473\), \(\sigma\approx -4.88)\) have imaginary part not 0.}
    \label{eta alpha}
\end{figure}


\section{Critical Strip}\label{sec eta critical}

There are several theorems that relate the Riemann hypothesis to the location of the zeros of the derivatives of $\zeta(s)$ (see, for example, \cite{s:5,lm:1}). 
Here we prove a new relation of this type for the function $\eta(s)$.
\begin{theorem}
The Riemann hypothesis implies $\eta'(s)\ne 0$ for $0< \sigma< 1/2$.
\end{theorem}

\begin{proof}
Combining the functional equation of \(\zeta(s)\) with (\ref{altern zeta}) and 
taking the logarithmic derivative, we get:
    \begin{align}
        -\frac{\eta'}{\eta}(1-s)&= -\log(2\pi) -\frac{1}{2}\pi \tan\left(\frac{\pi s}{2}\right) + \frac{\Gamma'}{\Gamma}(s)+\frac{\zeta'}{\zeta}(s)+\frac{2^{-s}\log 2}{(1-2^{-s})}.   \label{logarithmic eta}
    \end{align}
    Now, since for the Riemann zeta function we have (see \cite[(2.12.7)]{titchmarsh1986theory})
    \begin{align}
        \frac{\zeta'}{\zeta}(s)&= \log(2\pi)-1-\frac{\gamma}{2}- \frac{1}{s-1}-
        \frac{1}{2}\frac{\Gamma'}{\Gamma}\left(\frac{s}{2}+1\right)+\sum_{\rho}\left(\frac{1}{s-\rho}+\frac{1}{\rho}\right),
    \end{align}
    where the series runs over the complex roots of the zeta function and converges absolutely, we can write: 
    \begin{align}
        -\frac{\eta'}{\eta}(1-s)&=-\left(1+\frac{\gamma}{2}+\frac{1}{s-1}+\frac{1}{2}\pi \tan\left(\frac{\pi s}{2}\right)\right) +\frac{\log 2}{2^s-1}\notag\\
        &+ \frac{\Gamma'}{\Gamma}(s)-\frac{1}{2}\frac{\Gamma'}{\Gamma}\left(\frac{s}{2}+1\right)+\sum_{\rho}\left(\frac{1}{s-\rho}+\frac{1}{\rho}\right).
        \label{logarithmic eta long}
    \end{align}
    Since for a zero $\rho = a + bi$, we have
    \(
        \Re\left(\frac{1}{s-\rho}+\frac{1}{\rho}\right)=\frac{\sigma-a}{(\sigma-a)^2+(t-b)^2}+\frac{a}{a^2+b^2}>0,
    \)
    under the assumption of the Riemann hypothesis we can write: 
    \begin{align}\label{eta log total}
        -\frac{\eta'}{\eta}(1-s) >& \Re\left(\frac{\Gamma'}{\Gamma}(s)\right)-\Re\left(\frac{1}{2}\frac{\Gamma'}{\Gamma}\left(\frac{s}{2}+1)\right)\right)\notag\\
        &-\Re\left(1+\frac{\gamma}{2}+\frac{1}{s-1}+\frac{1}{2}\pi \tan\left(\frac{\pi s}{2}\right)\right) +\Re\left(\frac{\log 2}{2^s-1}\right).
    \end{align}
    In the plane cut along the non-positive real axis,
    \begin{align*}
        \frac{\Gamma'}{\Gamma}(s)= \log(s) -\frac{1}{2s}-\frac{1}{12s^2}+\int_0^{\infty}\frac{P_3(x)}{(s+x)^4}dx,
    \end{align*}
    where $P_3(x)$ is a function of period $1$ which is equal to 
    \(x^3-\frac{3}{2}x^2+\frac{1}{2}\) on $[0,1]$, and the $\log$ is principal. 
    We have $|P_3(x)|\leq \frac{1}{8},$ so
    \begin{align*}
        \int_0^{\infty}\frac{P_3(x)}{(s+x)^4}dx\leq \frac{1}{8}\int_0^{\infty}\frac{dx}{|s+x|^4}\leq \frac{1}{6|s|^3},
    \end{align*}
where the last inequality comes from \cite[(6)]{spira3}. Thus, 
    \begin{align*}
        \Re\left(\frac{\Gamma'}{\Gamma}(s)\right)\leq \log|s|-\frac{1}{2}|s|-\frac{1}{12}|s|^2-\frac{1}{6}|s|^3,
    \end{align*}
    and
    \begin{align*}
        \Re\left(\frac{\Gamma'}{\Gamma}\left(\frac{s}{2}+1\right)\right)\leq \log|\frac{s}{2}+1|+\frac{1}{|s+2|}+\frac{1}{3}|s+2|^2+\frac{4}{3}|s+2|^3.
    \end{align*}
    Next, note that 
    \begin{align*}
        \Re\left(\tan\left(\frac{\pi s}{2}\right)\right) \; \leq \; \left|\tan\left(\frac{\pi s}{2}\right) \right| \; \leq \; \frac{1+e^{-\pi|t|}}{1-e^{-\pi|t|}}
    \end{align*}
    by \cite[(14)]{spira3}, and this last function is monotonously decreasing with increasing $|t|$. Also,
    \begin{align*}
        \log\left(\frac{2|s|^2}{|s+2|}\right)=\log 2+\log|s|-\log \left|1+\frac{2}{s} \right|
    \end{align*}
    and
    \begin{align*}
        \log\left|1+\frac{2}{s}\right|\leq \log\left(1+\frac{2}{|s|}\right)\leq \frac{2}{|s|}.
    \end{align*}
    Finally, we have 
    \begin{align}
        \Re\left(\frac{1}{s-1}\right)\leq \left|\frac{1}{s-1}\right|\leq \frac{1}{|s|-1}.
    \end{align}
    Substituting all the results above 
    into (\ref{eta log total}), we obtain (valid for for all $|s|>1$): 
    \begin{align}\label{eta' total}
        -2\Re\left(\frac{\eta'}{\eta}(1-s)\right)>& \log|s| +\log 2 -2 -\gamma-\frac{2}{|s|}-\frac{\pi(1+e^{-\pi|t|})}{1-e^{-\pi|t|}}\notag\\
        &+\frac{2}{|s|-1}+\frac{\log 2}{2^{|s|}-1}-\frac{1}{|s|}-\frac{1}{6}|s|^2-\frac{1}{3}|s|^3\notag\\
        &- \frac{1}{|s+2|}-\frac{1}{3}|s+2|^2-\frac{4}{3}|s+2|^3.
    \end{align}
    For $|t|\geq 2$ and $|s|\geq 164,$ one can easily verify that the left hand side of (\ref{eta' total}) is greater than 0. Thus, assuming  the Riemann hypothesis, we have $\eta'(s)\ne 0$, for $0 < \sigma < \frac{1}{2}$ and $|t|\geq 164.$
\end{proof}


\section{Left Half Plane}\label{sec eta left}

We apply methods from \cite{yildirim} to find the locations of the zeros of \(\eta'\) on the left half plane. 

\begin{theorem}\label{thm deta bound}
Let \(n\ge 3\).  Then there exists \(\sigma\in[-2n-2,-2n-2+\epsilon]\) such that \(\eta'(\sigma)=0\) where
\begin{equation}\label{mainEpsIneq}
0 \leq \epsilon \leq -\frac{6}{\pi^2}(\log(2n+2)+\tau) + \frac{6}{\pi^2}\sqrt{(\log(2n+2)+\tau)^2+\frac{\pi^2}{3}}
\end{equation}
with $\tau=\frac{2}{3}\log(2) - \frac{1}{6} - \frac{1}{81} - \frac{1}{2}$.
\end{theorem}

We have $\tau\approx -2.05479129$ and the right hand side of (\ref{mainEpsIneq}) is less
than \(1.1\) for $n\ge 3$, and converges to 0 as \(n\) approaches infinity.  
Before we get to the proof of the theorem, we prove two lemmas.

\begin{lemma}\label{lemma:LogDeriv}
If $\sigma>3$, then $\left| \frac{\eta'}{\eta}(\sigma)\right|<\frac{1}{2}$.
\end{lemma}

\begin{proof}
Observe that $\eta'(s) = \sum\limits_{n=2}^{\infty}\frac{(-1)^{n+1}\log(n)}{n^s}$, for $\sigma>0$. And because the terms are monotonically decreasing for $\sigma>2$, and the continuous function is an upper bound for the corresponding discrete step-function, we get
\[
|\eta'(\sigma)|=\left| \sum\limits_{n=2}^{\infty}\frac{(-1)^{n+1}\log(n)}{n^\sigma}\right|\leq \sum\limits_{n=2}^{\infty}\frac{\log(n)}{n^\sigma} < \int\limits_{1}^{\infty}\frac{\log(x)}{x^\sigma}dx=\frac{1}{(\sigma-1)^2}.
\]
We now turn our attention to the series definition for $\eta(\sigma)$.  By applying the reverse triangle inequality and the sum-integral inequality, we obtain
\[
|\eta(\sigma)|=\left|\sum\limits_{n=1}^{\infty}\frac{(-1)^{n+1}}{n^\sigma}\right|\geq 1-\sum\limits_{n=2}^{\infty}\frac{1}{n^\sigma} > 1 - \int\limits_{1}^{\infty}\frac{1}{x^\sigma}dx = \frac{\sigma-2}{\sigma-1}.
\]
Since by assumption $\sigma>3$ with the above we get $\frac{1}{|\eta(\sigma)|}\leq\frac{\sigma-1}{\sigma-2}$.  Using this together with the inequality obtained earlier for $|\eta'(\sigma)|$, we have $\left|\frac{\eta'}{\eta}(\sigma)\right|\leq\frac{1}{(\sigma-1)(\sigma-2)}$.  Since $\sigma>3$ by assumption, we see that $\frac{1}{\sigma-1}<\frac{1}{2}$ and $\frac{1}{\sigma-2}<1$.  Thus, we have shown $\left| \frac{\eta'}{\eta}(\sigma) \right| < \frac{1}{2}$ as desired.
\end{proof}

\begin{lemma}\label{lemma:cotIneq}
Let $\sigma = 2n+3-\epsilon$, where $n\in\N$ and $0<\epsilon<1$. Then $$\frac{\pi}{2}\tan\left( \frac{\pi \sigma}{2}\right)\leq\frac{1}{\epsilon} - \frac{\pi^2}{12}\epsilon. $$
\end{lemma}

\begin{proof}
Using the sum-difference formulas for $\sin\left( \frac{\pi\sigma}{2} \right)$ and $\cos\left( \frac{\pi\sigma}{2} \right)$, we see that
$\tan\left( \frac{\pi \sigma}{2} \right) = \cot\left( \frac{\pi\epsilon}{2}\right)$. Recalling the Taylor series expansion of the cotangent function, we have
$\cot\left(\frac{\pi\epsilon}{2} \right) = \sum\limits_{n=0}^{\infty} \frac{(-1)^n 2^{2n}B_{2n}}{(2n)!} \left( \frac{\pi\epsilon}{2}\right)^{2n-1}$, where $B_{2n}$ denotes the $2n^{\text{th}}$ Bernoulli number, and $\zeta(2n) = \frac{(-1)^{n+1}B_{2n}(2\pi)^{2n}}{2(2n)!}$, for $n\geq 1$, by the well-known formula of Euler (see \cite{euler34}). This allows us to rewrite the sum as:  $\cot\left( \frac{\pi\epsilon}{2} \right)=\frac{2}{\pi\epsilon} - \frac{\pi\epsilon}{6} - \frac{2}{\pi}\sum\limits_{n=2}^{\infty}\zeta(2n)\left( \frac{\epsilon}{2}\right)^{2n-1}$. Multiplying by $\frac{\pi}{2}$, we get
$\frac{\pi}{2}\cot\left(\frac{\pi\epsilon}{2}\right)=\frac{1}{\epsilon}-\frac{\pi^2}{12}\epsilon - \sum\limits_{n=2}^{\infty}\zeta(2n)\left( \frac{\epsilon}{2} \right)^{2n-1}$, and since $\zeta(2n)>0$ for all $n\geq 1$ and $0 \leq \epsilon \leq 1$, the sum $\sum\limits_{n=2}^{\infty}\zeta(2n)\left( \frac{\epsilon}{2} \right)^{2n-1}$ always remains positive. From this
it immediately follows that $\frac{\pi}{2}\cot\left( \frac{\pi\epsilon}{2}\right) \leq \frac{1}{\epsilon} - \frac{\pi^2}{12}\epsilon$, as desired.
\end{proof}

\begin{figure}
   \centering
    \includegraphics[width=.9\textwidth]{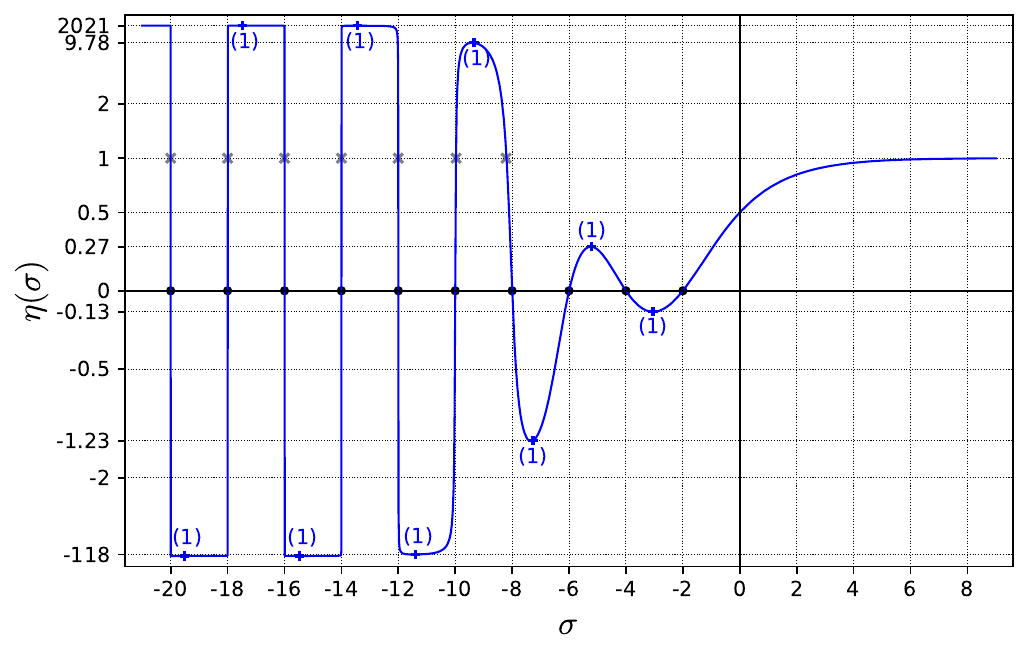}
    \caption{$\eta(\sigma)$ on \([-21,9]\) at an \(\arctan\) scale with zeros \(\bullet\) of \(\eta(s)\),
    zeros {\color{gray}\textsf{x}} of \(\eta(s)-1\), and zeros {\color{blue}$(1)^{\mathbf{+}}$} of \(\eta'(s)\) marked at the local extrema.}
    \label{fig eta neg}
\end{figure}

\begin{proof}[Proof of Theorem {\ref{thm deta bound}}]
Taking the logarithmic derivative of the functional equation (\ref{eta functional}) 
of \(\eta\)
in the form 
\[
\eta(1-s)= 2\left( \frac{1-2^s}{1-2^{1-s}} \right)(2\pi)^{-s}\Gamma(s)\cos\left( \frac{\pi s}{2}\right) \eta(s)
\]
and using elementary algebraic manipulation we obtain
\[
\eta'(1-s) = \textstyle\eta(1-s)\left[ \frac{\log(2)2^s}{1-2^s}\!+\!\frac{\log(2)2^{1-s}}{1-2^{1-s}}\!+\! \log(2\pi)\!-\!\psi(s)\!+\!\frac{\pi}{2}\tan\left( \frac{\pi s}{2} \right)\!-\!\frac{\eta'}{\eta}(s) \right].
\]
From this, we can now see that $\eta'(1-\sigma)=0$ if
\begin{equation}\label{eq:logDeriv}
\frac{\pi}{2}\tan\left( \frac{\pi \sigma}{2} \right) = \frac{\log(2)2^\sigma}{2^\sigma - 1} + \frac{\log(2)2^{1-\sigma}}{2^{1-\sigma}-1} - \log(2\pi) + \psi(\sigma) + \frac{\eta'}{\eta}(\sigma).
\end{equation}
With $\psi(\sigma)\geq \log(\sigma) - \frac{1}{2\sigma}-\frac{1}{12\sigma^2}$ for $\sigma\geq 1$ (see {\cite[equation 6.3.18 on page 259]{as1}}) and Lemma~\ref{lemma:LogDeriv} (assuming $\sigma>3$), equation ({\ref{eq:logDeriv}}) becomes
\begin{equation}\label{eq:logDeriv2}
\frac{\pi}{2}\tan\left( \frac{\pi \sigma}{2} \right) \geq \frac{\log(2)2^\sigma}{2^\sigma - 1} + \frac{\log(2)2^{1-\sigma}}{2^{1-\sigma}-1} - \log(2\pi) + \log(\sigma) - \frac{1}{2\sigma}-\frac{1}{12\sigma^2} - \frac{1}{2}.
\end{equation}
Note that if $\sigma>3$, then $\frac{\log(2)2^\sigma}{2^\sigma-1}\geq \log(2)$, $\frac{\log(2)2^{1-\sigma}}{2^{1-\sigma}-1} \geq -\frac{1}{3}\log(2)$, $\frac{1}{2\sigma}<\frac{1}{6}$, and $\frac{1}{12\sigma^2}<\frac{1}{81}$. Applying these inequalities to ({\ref{eq:logDeriv2}}) we obtain
\begin{equation}\label{eq:logDeriv3}
\frac{\pi}{2}\tan\left( \frac{\pi \sigma}{2}\right) \geq \log(\sigma) + \frac{2}{3}\log(2) - \log(2\pi) -\frac{1}{6} - \frac{1}{81} - \frac{1}{2}=\log(\sigma) +\tau
\end{equation}
Observe that the right hand side of ({\ref{eq:logDeriv3}}) tends to infinity as $\sigma$ gets large. The left hand side tends to infinity only when $\sigma$ is close to and to the right of odd integers.  This implies that there exists $\epsilon>0$ so that $\sigma=2n+3-\epsilon$.  It should also be noted that the right hand side of ({\ref{eq:logDeriv3}}) is positive for $\sigma>e^{-\tau}\approx 7.80520869$. This implies that $\epsilon<1$ when $\sigma>8$.  This is because if $\epsilon>1$ the left hand side of ({\ref{eq:logDeriv3}}) would be negative, yet the right would be positive, which is impossible.  We now attempt to approximate the value of $\epsilon$.  Applying Lemma \ref{lemma:cotIneq} to left hand side of the inequality in (\ref{eq:logDeriv3}), we have
\begin{equation}\label{eq:epsIneq}
\frac{1}{\epsilon} - \frac{\pi^2}{12}\epsilon \geq \log(2n+2) + \tau.
\end{equation}
Multiplying both sides of (\ref{eq:epsIneq}) by $\epsilon$ and pulling the right hand side over to the left, we see that (\ref{eq:epsIneq}) holds if
\begin{equation}\label{eq:epsIneq2}
-\frac{\pi^2}{12}\epsilon^2 - (\log(2n+2)+\tau)\epsilon + 1 \geq 0
\end{equation}
The left hand side of (\ref{eq:epsIneq2}) is a quadratic equation in $\epsilon$.  Using this fact, we see that (\ref{eq:epsIneq2}) holds whenever (\ref{mainEpsIneq}) holds.
\end{proof}
Analogously to the case of the Riemann \(\zeta\)-function, we have:

\begin{figure}
   \centering
    \includegraphics[width=.9\textwidth]{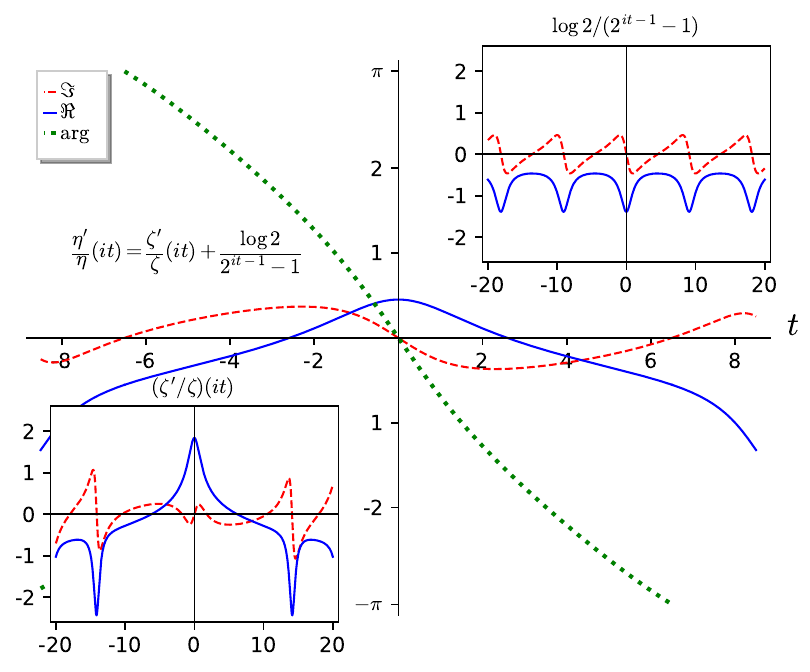}
    \caption{The real part (represented by a continuous line {\color{blue}--}), imaginary part (represented by a dashed line {\color{red}-\,\!-)}, and argument (represented by a dotted line
{\color{Green4}$\centerdot\!\centerdot\!\centerdot$}) of \(\frac{\eta'}{\eta}(it)=\frac{\zeta'}{\zeta}(it)+\frac{\log 2}{2^{it-1}-1}\) on the imaginary axis along with plots of the real and imaginary parts of \(\frac{\zeta'}{\zeta}(it)\) and \(\frac{\log 2}{2^{it-1}-1}\).}
    \label{fig eta arg}
\end{figure}

\begin{theorem}
The first derivative $\eta'(s)$ has only real zeros for $\sigma \leq 0,$ and exactly one real zero
in $(-2n-2,-2n)$ for  $n\in\N$.
\end{theorem}

\begin{proof} Let \(n\in\N\).
We proceed as in \cite[Proof of Theorem 9]{lm:1} and consider the logarithmic derivative 
\[
D(s):=
\frac{\eta'}{\eta}(s)=\left(\log\eta(s)\right)'
=\left(\log\left((1-2^{1-s})\zeta(s)\right)\right)'
=\frac{\log 2}{2^{s-1}-1}+\frac{\zeta'}{\zeta}(s)
\]
on the boundary \(\partial R\) of the rectangular region \(R=\{\sigma + it\mid -2n-1\le \sigma\le 0,\, -n\le t \le n\}\).
First, note that we have $\Re\left(\frac{\log 2}{2^{s-1}-1}\right)<0$ for \(s\in R\) and \(\Re(\frac{\zeta'}{\zeta}(s)) < 0\)  
on \(\partial R\) except for the stretch \(t\in (-r,r)\), where \(r\approx6.5052\).
The only local extremum of \(\Re(D(t))\) for \(t\in(-r,r)\) is the maximum at  
\(D(0)=\frac{\log 2}{2^{0-1}-1}+\frac{\zeta'}{\zeta}(0)=-2\log 2+\log(2\pi)>0\),
and moreover \(\Re(D(t))>0\) for \(t\in (-v,v)\), where \(v\approx 2.5841\),
and \(\Im(D(t))<0\) for \(t\in(-u,0)\), and \(\Im(D(t))>0\) for \(t\in(u,0)\),
where \(u\approx 3.5578\).
So \(\arg \frac{\eta'}{\eta}(s)\) changes by \(-2\pi\) on \(\partial R\) (see Figure \ref{fig eta arg}), and thus by the argument principle the number of zeros of \(\frac{\eta'}{\eta}(s)\) in \(R\) is one 
less than the number of its poles in \(R\).
But by (\ref{eta functional}) the only zeros of \(\eta(s)\) on the left half plane are at $s = -2n$ for $n\in\N$. Therefore, by Rolle's theorem, there is at least one zero of \(\eta'\) on $(-2n-2,-2n)$ and by Theorem \ref{thm deta bound} these are concentrated on the left half of this interval (for \(n\ge 4\)).
Hence \(\frac{\eta'}{\eta}(s)\) has \(n\) poles and \(n-1\) zeros in \(R\).

Since all this holds for all \(n\in\N\) the first derivative $\eta'(s)$ has only real zeros for $\sigma \leq 0,$ and exactly one real zero in $(-2n-2,-2n)$ for  $n\in\N$.
\end{proof}


\section{Counting Zeros}\label{sec eta number}


Let $N_\zeta(T)$ and $N_\zeta^k(T)$ denote the number of zeros $\rho$ inside the critical strip, with 
$0 < \Im(\rho) \leq T$, of $\zeta(s)$ and $\zetak(s)$, respectively.
Then by the classical Riemann-von Mangoldt formula (see Landau \cite{landau}): 
\[
N_\zeta(T) = \frac{T}{2 \pi} \log \frac{T}{2 \pi} - \frac{T}{2 \pi} + O(\log T),
\]
and according to a theorem of Berndt \cite{berndt}, for all $k\in\N$, 
\begin{equation}\label{eqberndt}
N_\zeta^k(T) =N_\zeta(T) - \frac{T}{2 \pi} \log 2 + O_k(\log T),
\end{equation}
Here we investigate the functions $N_{\eta}^{k}(T)$, that count the number of zeros
$\rho$ of $\eta^{(k)}(s)$ with $0 < \Im(\rho) \leq T$ where \(k\) is a positive integer.
By $N_{\eta}^{0}(T)$ we denote the number of zeros $\rho$ of $\eta(s)-1$ with $0 \leq \Im(\rho) \leq T$.
We prove: 

\begin{theorem}\label{thm eta number}
    For $k\in\N\cup\{ 0\}$, as $T\to \infty$, we have
    \(N_{\eta}^{k}(T) = N_{\zeta}(T)+\mathcal{O}(\log T)\).
\end{theorem}

\begin{proof}
    The method is standard. We count the number of zeros of  $\eta^{(k)}(s)$ inside rectangles with carefully chosen sides, making sure that  the right vertical $\sigma_{k}$ is large enough so that there are no zeros for $\sigma \geq \sigma_{k}>1$, and the left vertical  $\beta_{k}$ is sufficiently negative so that there are no zeros for $\sigma\leq \beta_{k}$. Also, we look for $\sigma_{k}$ large enough so that the first term dominates the series, i.e. 
    \begin{equation}\label{bound T1}
        \sum_{n=3}^{\infty}\frac{\log^{k} n}{(n/2)^{\sigma_{k}}} \; \leq \; \frac{1}{2}\log^{k} 2, 
    \end{equation}
    because then, for $\sigma \geq \sigma_{k}, $ we have 
\begin{equation}\label{eta left frac}
        |\eta^{(k)}(s)|\geq\frac{\log^{k} 2}{2^{\sigma}}-\sum_{n=3}^{\infty}\frac{\log^{k} n}{n^{\sigma}}\geq\frac{\log^{k} 2}{2^{\sigma}}-\frac{1}{2}\frac{\log^{k} 2}{2^{\sigma}}=\frac{1}{2}\frac{\log^{k} 2}{2^\sigma}>0. 
\end{equation}
Choose $\tau_{k}>0$ so that $\eta^{(k)}(s)$ has no zeros for $0<t\leq \tau_{k}.$ Lastly, choose $T_{k}=T$ so that the line $t=T$ is free of zeros for $\eta^{(k)}(s)$. Let $C$ be the rectangle (described positively) with vertices
\[\beta_{k}+i\tau_{k}, \sigma_{k}+i\tau_{k}, \sigma_{k}+iT, \beta_{k}+iT.\]
By the argument principle, 
\begin{align*}
    N_{\eta}^{k}(T) &= \frac{1}{2\pi i}\int_C \frac{d}{ds}\log(\etak) \; ds\\
    &=\frac{1}{2\pi i}\left\{\int_{\beta_{k}+i\tau_{k}}^{\sigma_{k}+i\tau_{k}}+\int_{\sigma_{k}+i\tau_{k}}^{\sigma_{k}+iT}+\int_{\sigma_{k}+iT}^{\beta_{k}+iT}+\int_{\beta_{k}+iT}^{\beta_{k}+i\tau_{k}}\right\}\frac{d}{ds}\log(\etak) \; ds\\
    &=\frac{1}{2\pi i}\{I_1+I_2+I_3+I_4\},
\end{align*}
say. We examine $I_1,I_2,I_3$ and $I_4$ in turn. 

Firstly, $I_1$ is independent of $T.$ Hence, $I_1=\mathcal{O}(1)$.
Secondly, 
\begin{align*}
    I_2&=\left[\log(\eta^{(k)}(s))\right]_{\sigma_{k}+i\tau_{k}}^{\sigma_{k}+iT} = \left[\log\left(\frac{(-1)^{k+1}\log 2}{2^s}\right)\right]_{\sigma_{k}+i\tau_{k}}^{\sigma_{k}+iT}+\left[\log(1+g(s))\right]_{\sigma_{k}+i\tau_{k}}^{\sigma_{k}+iT}
\end{align*}
where the function $g(s)$ is defined as 
\begin{align*}
    g(s)=\sum_{n=3}^{\infty}\frac{(-1)^{n-2}(\log(n)/\log 2)^{k}}{(n/2)^s}.
\end{align*}
By (\ref{bound T1}), $|g(s)|\leq \frac{1}{2}$ on the line $\sigma=\sigma_{k}$. Hence, $\Re\{1+g(s)\}\geq \frac{1}{2}$ and the argument of $1+g(s)$ ranges over an interval length no greater than $\pi$ as $s$ traverses the line $\sigma=\sigma_{k}.$ Hence, 
\(I_2 =-iT\log 2+\mathcal{O}(1)\).

To estimate $I_3$ we put $ \phi_{k}(s) = (-1)^{k}e^{iT\log 2}\eta^{(k)}(s)$. 
Hence, the leading term of the Dirichlet series for $\phi_{k}(s)$ is positive at $s=\sigma_{k}+iT$. 
Now, if $q$ denotes the number of zeros of $\Re\{\cdot\}$ on $J=(\beta_{k}+iT, \sigma_{k}+iT),$ it is divided into at most $q+1$ subintervals in each of which $\Re\{\cdot\}$ is of constant sign. Hence, the variation of $\Im\{\log(\phi_{k}(s)\}=\arg\{\phi_{k}(s)\}$ is at most $\pi$ in each subinterval.
\[\Im\{I_3\}=|\Im\{[\log(\phi_{k}(s)]_{\sigma_{k}+iT}^{\beta_{k}+iT}\}|\leq (q+1)\pi.\]
To estimate $q$ we first let 
\(f(z) = \frac{1}{2}\{\phi_{\sigma}(z+iT)+\overline{\phi_{\sigma}(\overline{z}+iT)}\}\)
and note that if $z=\sigma$ is real,
\(f(\sigma)=\Re\{\phi_{\sigma}(\sigma+iT)\}\).
Choose $T$ large enough so that 
\[T>\tau_{k}+2(\sigma_{k}-\beta_{k}).\]
Let $R=T-\tau_{k}$ and consider those $z$ with $|z-\sigma_{k}|<R$. Then 
\[\Im\{z+iT\}>T-R=\tau_{k}>0.\]
Thus, $\phi_{\sigma}(z+iT)$, and therefore $f(z)$ is analytic for $|z-\sigma_{k}|<R$. Let $n(\rho)$ denote the number of zeros of $f(z)$ in the circle $|z-\sigma_{k}|\leq \rho.$ If $r=2(\sigma_{k}-\beta_{k})$ and $r_1=\frac{1}{2}r$, we have
\[\int_0^r\frac{n(\rho)}{\rho}d\rho\geq n(r_1)\int_{r_1}^r\frac{d\rho}{\rho}=n(r_1)\log 2.\]
From Jensen's formula \cite{Jensen}, we have the following:
\begin{align}\label{eta n(r)}
    n(r_1)\leq \frac{1}{2\pi\log 2}\int_0^{2\pi}\log|f(re^{i\theta}+\sigma_{k})|d\theta-\frac{1}{\log 2}\log|f(\sigma_{k})|
\end{align}
Since $\etalpha(s)=\mathcal{O}(t^A)$ as $t \to \infty$ and 
\begin{align*}
    f(\sigma_{k})&= \Re\{\phi_{k}(\sigma_{k}+iT)\} =\Re\left\{\frac{\log^{k}(2)}{2^{\sigma_{k}}}+\frac{-1}{2^{\sigma_{k}}}\sum_{n=3}^{\infty}\frac{(-1)^{n-1}\log^{k}(n)}{(n/2)^{\sigma_{k}+iT}}\right\}\\
    &\geq \frac{\log^{k}(2)}{2^{\sigma_{k}}}+\frac{-1}{2^{\sigma_{k}}}\sum_{n=3}^{\infty}\frac{(-1)^{n-1}\log^{k}(n)}{(n/2)^{\sigma_{k}}} \; \geq \; \frac{1}{2}\frac{\log^{k}(2)}{2^{\sigma_k}}
\end{align*}
by (\ref{bound T1}), it follows from (\ref{eta n(r)}) that $n(r_1)=\mathcal{O}(\log T).$ Now, the zeros of $\Re\{\sigma_k+iT\}$ on $J$ correspond to an equal number of zeros of $f(z)$ on $(\beta_{k},\sigma_{k}).$ since $r_1=\sigma_{k}-\beta_{k}$, $(\beta_{k},\sigma_{k})$ is contained in the disc $|z-\sigma_{k}|\leq r_1.$ Hence $q\leq n(r_1)$ and 
\[\Im\{I_3\}=\mathcal{O}(\log T).\]
Lastly,
\[I_4=[\log(\eta^{(k)}(s)]_{\beta_{k}+iT}^{\beta_{k}+i\tau_{k}}\]
From the functional equation for $\eta$ (\ref{eta functional}) and Leibniz' Rule, 
\begin{align}\label{eta k div}
    \eta^{(k)}(s) =& \left\{(2-2^{s})\pi^{s-1}(-s)\sin\left(-\frac{s\pi}{2}\right)\Gamma(-s)\zeta(1-s)\right\}^{(k)}\notag\\ 
    =&\left\{(2-2^{s})\pi^{s-1}s\sin\left(\frac{s\pi}{2}\right)\Gamma(-s)\right\}^{(k)}\zeta(1-s)\notag\\ 
    &+ \sum_{j=0}^{k-1}\binom{k}{j}\left\{(2-2^{s})\pi^{s-1}s\sin\left(\frac{s\pi}{2}\right)\Gamma(-s)\right\}^{(j)}\zeta^{(k-j)}(1-s).
\end{align}
We set 
\begin{align*}
    \mu_k = \left\{(2-2^{s})\pi^{s-1}(-s)\sin\left(-\frac{s\pi}{2}\right)\Gamma(-s)\right\}^{(k)}
\end{align*}
and 
\begin{align*}
    \mu_{k-j}=\sum_{j=0}^{k-1}\binom{k}{j}\left\{(2-2^{s})\pi^{s-1}s\sin\left(\frac{s\pi}{2}\right)\Gamma(-s)\right\}^{(j)}
\end{align*}
%
%
Therefore, by (\ref{eta k div}) we have for the first derivative that 
\begin{align*}
    \eta^{(k)}(s) &= \mu_k\zeta(1-s)+\mu_{k-j}\zeta'(1-s)\\
    &=\mu_k+\mu_k(\zeta(1-s)-1)+\mu_{k-j}\zeta^{(k)}(1-s).
\end{align*}
Some important things to note are for sufficiently negative values $s$
\begin{align*}
    \left|\frac{1}{2-2^s}\right|\leq 1 \text{ and } 
    \left|\frac{1}{s}\right|\leq 1.
\end{align*}
We look at the following derivatives
\begin{align}\label{pi div}
    \{\pi^s\}^{(j)} = \log(\pi)^j(\pi)^s
\text{ and }
    \left(\sin\left(\frac{s\pi}{2}\right)\right)^{(j)}=\pm \left(\frac{\pi}{2}\right)^{j}\left\{\substack{\sin(\frac{s\pi}{2})\\\cos(\frac{s\pi}{2})} \right\},
\end{align}
depending whether $j$ is even or odd. By Stirling's formula for $\Gamma(s)$, we have 
\begin{align}\label{Gamma div}
    \Gamma^{(j)}(s) = \Gamma(s)\left\{\log^j(s)+\sum_{n=0}^{j-1}E_{nj}(s)log^n(s)\right\},
\end{align}
where $E_{nj(s)}=\mathcal{O}(1/s)$. From (\ref{pi div}) and (\ref{Gamma div}) we see that we can write (\ref{eta k div}) in the form
\begin{align}\label{eta k R12 div}
    \eta^{(k)}(s) = (2-2^{s})(\pi)^se^{\frac{is\pi}{2}}\Gamma(-s)\{R_1(s)+R_2(s)\},
\end{align}
where 
$\textstyle
    R_1(s) = \frac{\mu_{k}}{(2-2^{s})(\pi)^se^{\frac{is\pi}{2}}\Gamma(-s)}
\text{ and }
    R_2(s) =\frac{\mu_k(\zeta(1-s)-1)+\mu_{k-j}\zeta^{k}(s-1)}{(2-2^{s})(\pi)^se^{\frac{is\pi}{2}}\Gamma(-s)}
.$

Both $R_1(s)$ and $R_2(s)$ are finite sums, each term of which is $\mathcal{O}(\log^k(s))$ on $J'=(\beta_k+iT,\beta_k+i\tau_k)$. From (\ref{Gamma div}), we have 
\begin{align*}
    \left((2-2^{s})\pi^{s-1}s\sin\left(\frac{s\pi}{2}\right)\Gamma(-s)\right)^{(k)}\!\!\!&= \sum\limits_{j=0}^{k}\binom{k}{j}\left((2-2^{s})\pi^{s-1}s\sin\left(\frac{s\pi}{2}\right)\right)^{(k-j)}\\\cdot\;&\Gamma(-s)\left(\log^j(-s)+\sum\limits_{n=0}^{j-1}E_{nj}(s)\log^n(-s)\right)
\end{align*}
where $E_{nj(s)}=\mathcal{O}(1/s)$, we see from (\ref{pi div}) and (\ref{eta k R12 div}) that for $\beta_k$ sufficiently negative, $R_1(s)$ is dominated by $\log^k(-s)$ and is bounded away from zero. Also, choose $\beta_k$ sufficiently negative so that (\ref{eta left frac}) holds and so that 
\begin{align*}
    \left|\frac{R_2(s)}{R_1(s)}\right|
    &=\left|\frac{\mu_{k}(\zeta(1-s)-1)+\mu_{k-j}\zeta^{k}(s-1)}{(2-2^{s})(\pi)^se^{\frac{is\pi}{2}}\Gamma(-s)}\Bigm/\frac{\mu_{k}}{(2-2^{s})(\pi)^se^{\frac{is\pi}{2}}\Gamma(-s)}\right|\\
    &=\left|\frac{\mu_k(\zeta(1-s)-1)+\mu_{k-j}\zeta^{(k)}(s-1)}{\mu_k}\right|\\ 
    &=  \left|\zeta(1-s)-1+\frac{\mu_{k-j}}{\mu_{k}}\zeta^{(k)}(1-s)\right| \; < \; 1, 
\end{align*}
where $s$ belongs to the segment $J'$. Thus, $\arg\left\{1+\frac{R_2(s)}{R_{1}(s)}\right\}$ varies over an interval of length no greater than $\pi$ as $s$ traverses $J'$. Hence, from (\ref{eta k R12 div}), 
\begin{align*}
    I_4 = &\Bigg[\log(2-2^s)+\log(\pi)^s+\log(s)-\frac{i s\pi}{2}+\log(\Gamma(-s))\\ 
    &\;\;+\log(R_1(s))+\log\left\{1+\frac{R_2(s)}{R_1(s)}\right\}\Bigg]_{\beta_k+iT}^{\beta_k+i\tau_k}\\
    =&\;\mathcal{O}(1)-iT\log(\pi) +\mathcal{O}(1)+\mathcal{O}(\log T)-\frac{T\pi}{2}+\mathcal{O}(1) \\
    &+\left[-(s+\frac{1}{2})\log(-s)+s+\mathcal{O}(1) \right]_{\beta_k+iT}^{\beta_k+i\tau_k} 
    +\;\;\mathcal{O}\{\log(\log T)\}+\mathcal{O}(1),
\end{align*}
upon the use of Stirling's formula for $\log(\Gamma(s))$. Now, 
\begin{align*}
    (\alpha_k+\frac{1}{2}+iT)\log(-\alpha_k-iT)&=(\alpha_k+\frac{1}{2}+iT)\log(-iT(1-\alpha_k)/iT)\\
    &=(\alpha_k+\frac{1}{2}+iT)\log(-iT)+\mathcal{O}(1)\\
    &=iT\log T +\frac{1}{2}\pi T +\mathcal{O}(\log(T)).
\end{align*}
Thus, \(I_4 = iT(\log T-\log(\pi)-1)+\mathcal{O}(\log T)\) and we conclude 
\begin{align*}
    N^{k}_{\eta}(T) &= \frac{1}{2\pi}\sum_{j=1}^4\Im\{I_j\} =\frac{1}{2\pi}\{T\log 2+T(\log T-\log(\pi)-1)\}+\mathcal{O}(\log T)\\
    &=N^{k}_{\zeta}(T)+\frac{T\log 2}{2\pi}+\mathcal{O}(\log T) =N_{\zeta}(T)+\mathcal{O}(\log T),
\end{align*} 
which proves the theorem.
\end{proof} 

\begin{note}
Theorem \ref{thm eta number} is compatible with the observation that there appears to be a one to one correspondence between the zeros of \(\zeta\) and the zeros of the derivatives of \(\eta\) established by the
paths of zeros of \(\eta(s)-c\) and the paths of zeros of \(\etalpha(s)\).  There is no such correspondence
for the zeros \(s\) of \(\eta\) with \(\Re(s)=1\) that are not connected to the path of zeros of the derivatives by
a path of zeros of \(\eta(s)-c\) (see Note \ref{note -c}).
\end{note}


\section{Miscellaneous Conjectures and Remarks}\label{sec eta conjecture}

We would like to conclude this paper with some worthwhile observations, open problems, and unproved assertions we believe to be likely to be true.

From the horizontal alignment of the regions described in Theorems \ref{thmone} and \ref{thmboxzeroeta} 
and the paths of zeros from Corollary \ref{cor curve eta} and the numerical evidence illustrated 
in Figure \ref{eta plot} it seems reasonable to conjecture that the situation remains largely unchanged for the zero-counting functions $N_{\eta}^{\alpha}(T)$ for the fractional derivatives $\eta^{(\alpha)}(s)$:

\begin{conjecture}\label{conj count frac}
For $\alpha\in (0,\infty)$, as $T\to \infty$, we have 
\(N_{\eta}^{\alpha}(T)=N_{\zeta}(T)+\mathcal{O}(\log T). \)
\end{conjecture}

In Figure \ref{fig eta neg} we see that there cannot be a path consisting of real zeros of    
\(\eta(\sigma)-c\) for \(c\in[0,1]\) from the zeros \(-4\) and \(-6\) to a zero of \(\eta(\sigma)-1\).
In Figures \ref{eta alpha} and \ref{fig eta double zero} we see how these paths enter the complex
plane to reach zeros of \(\eta(s)-1\) at 
\(s\approx-5.5169183502\pm1.2011878887i\).

Following the paths of the zeros of \(\eta(s)^{(\alpha)}=0\) from these zeros of \(\eta(s)-1\) we 
find a double zero of \(\eta^{\alpha}\) for  \(\alpha\approx 0.4734296409\) near 
\(z=-4.8804424428\) (see again Figures \ref{eta alpha} and \ref{fig eta double zero})
where these paths of zeros of fractional derivatives meet the real axis.

The existence of this double zero is  
quite unique in the theory of the Riemann zeta function and Dirichlet L-functions.
Unfortunately, as surprising as the existence of this unusual zero is, as of right 
now we do not have a sufficient rational justification for it.  
We note that it is the only multiple zero of a fractional derivative 
for any of these functions that we have been able to discover so far. 

\begin{figure}[ht]
    \centering   
    \includegraphics[width=0.8\textwidth]{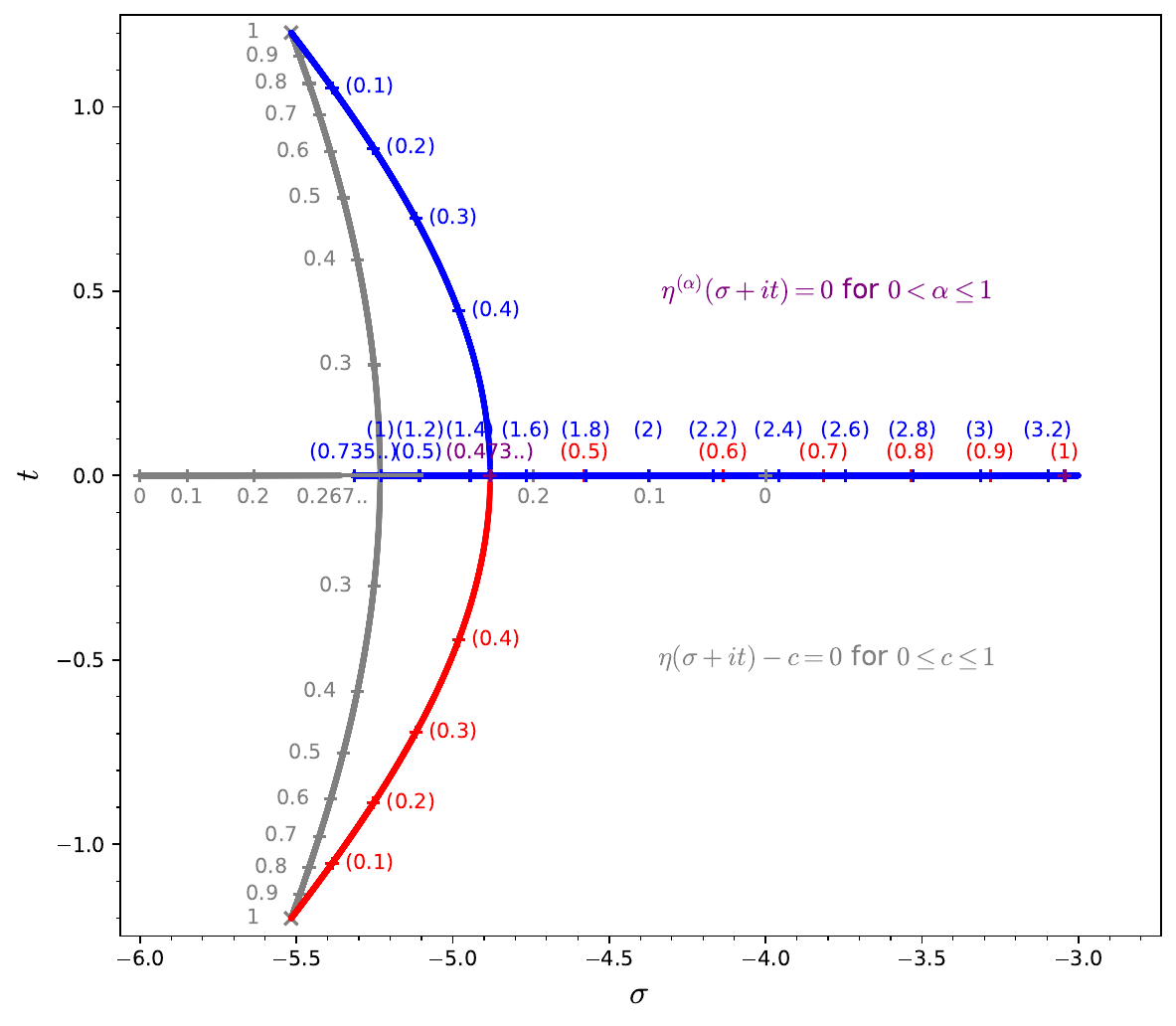}
    \caption{Zeros  of the fractional derivatives $\etalpha$ of the Euler \(\eta\)-function  on the left half plan with a double zero for \(\alpha\approx 0.473\) near \(\sigma = -4.88\) with selected zeros represented by $\color{purple}{+}^{(\alpha)}$. 
    The zeros  of \(\eta\) 
    at \(\sigma=-4\) and \(\sigma=-6\) denoted $\color{gray}\stackrel{{+}}{0}$ are connected to the 
     zeros  of $\eta(s)-1$ denoted {\color{gray}1\textsf{x}} by the zeros  of \(\eta(\sigma+it)-c\) for \(0\le c\le 1\) denoted $\color{gray}c^{+}$.   There is a double zero 
    of \(\eta(\sigma+it)-c\) for \(c\approx 0.267\) near \(\sigma=-5.23\).
    }
    \label{fig eta double zero}
\end{figure}

Figure \ref{eta alpha} also gives an idea how the paths of zeros of the (fractional) derivatives from the negative
real axes continue on the positive real axis on the paths from Corollary \ref{cor curve eta} for \(j=-1\).
All of the above supports a conjecture made by Ricky Farr in 2022:

\begin{conjecture}
If $k\in\N$ and $\sigma+it\in\C$ with $\sigma<0$ and $t\ne0$ then 
\(\eta^{(k)}(\sigma+it)\ne 0\).
\end{conjecture}

We believe that this extends to fractional derivatives with the exception of the zeros shown in 
Figure \ref{fig eta double zero}.

\begin{note}
This differs from the situation for the Riemann zeta function for which 
it has been proven that there is exactly one conjugate pair of zeros for
each of \(\zeta''\) and \(\zeta'''\) \cite{yildirim} on the left half plane
and it has been observed that the number of zeros on the left half plane of \(\zetak\) grows with \(k\). 
See \cite{fp} for examples of more zeros of integral derivatives and \cite{fps-icms} for counts of these along with examples of zeros of fractional derivatives.
\end{note}


\section{Acknowledgments}

The authors thank Ricky Farr for sharing his observations about the distribution of derivatives of \(\eta\) with them.  Most of the results discussed in this paper come from the research work that was carried out for the doctoral dissertation of the first-named author  \cite{c24} at the University of North Carolina Greensboro.  

All computations and plots in this paper were performed with the computer algebra system Sage \cite{sage}
using the mpmath multi-precision library \cite{mpmath}.
We evaluate fractional derivatives of \(\eta\) with 
an adaptation of the algorithm for the evaluation of the algorithm for the evaluation of fractional derivatives of \(\zeta\) from \cite{fps-icms}.
Our implementation in Sage along with examples that compute and plot 
paths of zeros of fractional derivatives and code for finding the double zero discussed in the paragraphs after Conjecture
\ref{conj count frac} is available at \url{https://github.com/sebastian-pauli/eta}.


\bibliography{zeta}

@article{abel1823,  
author = {Abel, N. H.},
title = {Solution de quelques probl{\' e}mes {\`a} l’aide d’int{\' e}grales d{\' e}finies}, 
journal ={Mag. Naturvidenskaberne}, 
issue = {2}, 
year = {1823} ,
note = {see also Abel’s Oeuvres, Vol. 1, 11--27, Christiania, 1881}
}

@book {as1,
    AUTHOR = {Abramowitz, M. and Stegun, I. A.},
     TITLE = {Handbook of mathematical functions with formulas, graphs, and
              mathematical tables},
    SERIES = {National Bureau of Standards Applied Mathematics Series},
    VOLUME = {55},
 PUBLISHER = {Dover},
      YEAR = {1970},
}

@article {AK15,
    AUTHOR = {Alzer, H. and Kwong, M.K.},
     TITLE = {On the concavity of {D}irichlet's eta function and related
              functional inequalities},
   JOURNAL = {J. Number Theory},
  FJOURNAL = {Journal of Number Theory},
    VOLUME = {151},
      YEAR = {2015},
     PAGES = {172--196},
      ISSN = {0022-314X,1096-1658},
   _MRCLASS = {11M41 (26D07 39B62)},
  _MRNUMBER = {3314208},
_MRREVIEWER = {Renata\ Macaitien\.e},
       DOI = {10.1016/j.jnt.2014.12.009},
       URL = {https://doi.org/10.1016/j.jnt.2014.12.009},
}

@PhdThesis{c24,
  Title                    = {On zeros of fractional derivatives of {D}irichlet series and polynomials},
  Author                   = {Caparatta, T.},
  School                   = {University of North Carolina Greensboro},
  Year                     = {2024},
  URL                      = {https://libres.uncg.edu/ir//list-etd.aspx?styp=ke&alpha=d&bs=Dirichlet%20Series}
}

@article {BF21,
    AUTHOR = {Boyadzhiev, K. and Frontczak, R.},
TITLE = {Series involving {E}uler's eta (or {D}irichlet eta) function},
   JOURNAL = {J. Integer Seq.},
  FJOURNAL = {Journal of Integer Sequences},
    VOLUME = {24},
      YEAR = {2021},
    NUMBER = {9},
     PAGES = {Art. 21.9.1, 14},
      ISSN = {1530-7638},
   _MRCLASS = {11F20 (11B39)},
  _MRNUMBER = {4336093},
}

@article {bps,
    AUTHOR = {Binder, T. and Pauli, S. and Saidak, F.},
     TITLE = {Zeros of high derivatives of the {R}iemann zeta function},
   JOURNAL = {Rocky Mountain J. Math.},
  FJOURNAL = {The Rocky Mountain Journal of Mathematics},
    VOLUME = {45},
      YEAR = {2015},
    NUMBER = {3},
     PAGES = {903--926},
      ISSN = {0035-7596},
   _MRCLASS = {11M06 (11M26)},
  _MRNUMBER = {3385969},
_MRREVIEWER = {Kamel Mazhouda},
       URL = {https://doi.org/10.1216/RMJ-2015-45-3-903},
}

@incollection{euler1730, 
author = {Euler, L.},
title = {Letter to {G}oldbach, {J}anuary 8, 1730}, 
editor ={Fuss, P. H.},
booktitle = {Correspondance math{\' e}matique et physique ... du XVIII. si{\' e}cle, l’Academie imperiale des sciences, St.-P{\' e}tersbourg}, 
volume = {1},
year = {1843}
}

@article{euler34,
author={Euler, L.}, 
title ={De progressionibus harmonicis observationes}, 
journal={Comm. Acad. Sci. Petropol},
volume={7}, 
pages={150-–161},
year={1740},
note={also see Euler’s Opera Omnia, Series 1, Vol. 14, 87–100}
}

@article{Euler1749, 
author={Euler, L.},
title ={Remarques sur un beau rapport entre les s{\' e}ries des puissances tant directes que r{\' e}ciproques}, 
journal = {Opera Omnia}, 
series = {1}, 
volume = {15},
pages = {70--90},
year = {1749},
}

@incollection {fp,
    AUTHOR = {Farr, R. E. and Pauli, S.},
    TITLE = {More Zeros of the Derivatives of the {R}iemann Zeta Function on the Left Half Plane},
    BOOKTITLE = {Topics form the 8th Annual UNCG  Regional Mathematics and Statistics Conference},
    SERIES = {Springer Proceedings in Mathematics \& Statistics},
    PAGES = {93--104},
    PUBLISHER = {Springer},
    YEAR = {2013}
}

@article {fps,
    AUTHOR = {Farr, R. E. and Pauli, S. and Saidak, F.},
     TITLE = {On fractional {S}tieltjes constants},
   JOURNAL = {Indag. Math. (N.S.)},
  FJOURNAL = {Koninklijke Nederlandse Akademie van Wetenschappen.
              Indagationes Mathematicae. New Series},
    VOLUME = {29},
      YEAR = {2018},
    NUMBER = {5},
     PAGES = {1425--1431},
      ISSN = {0019-3577},
   _MRCLASS = {11M35},
  _MRNUMBER = {3853435},
_MRREVIEWER = {Mattia Righetti},
       URL = {https://doi.org/10.1016/j.indag.2018.07.005},
}

@article {fps2,
    AUTHOR = {Farr, R. E. and Pauli, S. and Saidak, F.},
     TITLE = {A zero free region for the fractional derivatives of the {R}iemann zeta function},
   JOURNAL = {NZJM},
  FJOURNAL = {New Zealand Journal of Mathematics},
    VOLUME = {50},
      YEAR = {2018},
     PAGES = {1--9},
      ISSN = {1179-4984},
       URL = {http://nzjm.math.auckland.ac.nz/index.php/A_zero-free_region_for_the_fractional_derivatives_of_the_Riemann_zeta_function}
}

@article{hardy,
 ISSN = {09093540, 2446077X},
 URL = {http://www.jstor.org/stable/24529536},
 author = {Hardy, G.H.},
 journal = {Matematisk Tidsskrift. B},
 pages = {71--73},
 publisher = {Mathematica Scandinavica},
 title = {A new proof of the functional equation for the Zeta-function},
 urldate = {2024-01-02},
 year = {1922}
}

@book{titchmarsh1986theory,
  title={The Theory of the {Riemann} Zeta-function},
  author={Titchmarsh, E.C. and Heath-Brown, D.R.},
  isbn={9780198533696},
  lccn={lc86012520},
  series={Oxford science publications},
  url={https://books.google.com/books?id=1CyfApMt8JYC},
  year={1986},
  publisher={Clarendon Press}
}

@article {g:1867,
title={{\"U}ber `begrenzte' {D}erivation und deren {A}nwendung},
  author={Gr{\"u}nwald, A.K.},
  journal={Z. angew.  Math. und Phys},
  volume={12},
  pages={441--480},
  year={1867}
}

@article {kreminski03,
    AUTHOR = {Kreminski, R.},
     TITLE = {Newton-{C}otes integration for approximating {S}tieltjes
              (generalized {E}uler) constants},
   JOURNAL = {Math. Comp.},
  FJOURNAL = {Mathematics of Computation},
    VOLUME = {72},
      YEAR = {2003},
    NUMBER = {243},
     PAGES = {1379--1397 (electronic)},
      ISSN = {0025-5718},
     CODEN = {MCMPAF},
   _MRCLASS = {11Y60 (11M35)},
  _MRNUMBER = {1972742},
_MRREVIEWER = {Ekatherina A. Karatsuba},
       URL = {http://dx.doi.org/10.1090/S0025-5718-02-01483-7},
}

@book{le-1809,
author={Legendre, A.M.},
title= {M{\' e}moires de la classe des sciences math. et physiques},
publisher = {l’Institut de France}, 
address = {Paris},
year = {1809}
}

@article {l:1869-1,
    AUTHOR = {Letnikov, A.V.},
     TITLE = {Theory of differentiation of fractional order},
   JOURNAL = {Mat. Sbornik},
    VOLUME = {3},
     PAGES = {1--68},
      YEAR = {1868},
}

@article {l:1869-2,
    AUTHOR = {Letnikov., A.V.},
     TITLE = {Historical development of the theory of differentiation of fractional order},
   JOURNAL = {Mat. Sbornik},
    VOLUME = {3},
     PAGES = {85--119},
      YEAR = {1868},
}

@article {lm:1,
    AUTHOR = {Levinson, N. and Montgomery, H.L.},
     TITLE = {Zeros of the derivatives of the {R}iemann zeta-function},
   JOURNAL = {Acta Math.},
  FJOURNAL = {Acta Mathematica},
    VOLUME = {133},
      YEAR = {1974},
     PAGES = {49--65},
      ISSN = {0001-5962},
   _MRCLASS = {10H05},
  _MRNUMBER = {417074},
_MRREVIEWER = {Bruce C. Berndt},
       URL = {https://doi.org/10.1007/BF02392141},
}

@article {Mil13,
    AUTHOR = {Milgram, M.S.},
     TITLE = {Integral and series representations of {R}iemann's zeta
              function and {D}irichlet's eta function and a medley of
              related results},
   JOURNAL = {J. Math.},
  FJOURNAL = {Journal of Mathematics},
      YEAR = {2013},
     PAGES = {Art. ID 181724, 17},
      ISSN = {2314-4629,2314-4785},
   _MRCLASS = {11M06},
  _MRNUMBER = {3100743},
       DOI = {10.1155/2013/181724},
       URL = {https://doi.org/10.1155/2013/181724},
}

@manual{mpmath,
  key     = {mpmath},
  author  = {Johansson F. and others},
  title   = {mpmath: a {P}ython library for arbitrary-precision floating-point arithmetic (version 1.4)},
  note    = {\url{https://mpmath.org}},
  year    = {2026}
}

@article {zero-free-fract,
    AUTHOR = {Pauli, S. and Saidak, F.},
     TITLE = {Zero-free regions of the fractional derivatives of the
              {R}iemann zeta function},
   JOURNAL = {Lith. Math. J.},
  FJOURNAL = {Lithuanian Mathematical Journal},
    VOLUME = {62},
      YEAR = {2022},
    NUMBER = {1},
     PAGES = {99--112},
      ISSN = {0363-1672},
   _MRCLASS = {11M06 (26A33)},
  _MRNUMBER = {4383337},
       DOI = {10.1007/s10986-022-09551-2},
       URL = {https://doi.org/10.1007/s10986-022-09551-2},
}

@manual{sage,
  Key          = {SageMath},
  Author       = {{Sage Developers}},
  Title        = {{S}age{M}ath {M}athematical {S}oftware {S}ystem},
  note         = {\url{https://www.sagemath.org}},
  Year         = {2023},
}

@article {s:3,
    AUTHOR = {Skorokhodov, S.L.},
     TITLE = {Pad\'{e} approximants and numerical analysis of the {R}iemann zeta
              function},
   JOURNAL = {Zh. Vychisl. Mat. Mat. Fiz.},
  FJOURNAL = {Zhurnal Vychislitel\cprime no\u{\i} Matematiki i Matematichesko\u{\i} Fiziki.
              Rossi\u{\i}skaya Akademiya Nauk},
    VOLUME = {43},
      YEAR = {2003},
    NUMBER = {9},
     PAGES = {1330--1352},
      ISSN = {0044-4669},
   _MRCLASS = {11Y35 (11M06 41A21 65D20)},
  _MRNUMBER = {2014985},
_MRREVIEWER = {Matti Jutila},
}

@article {Son03,
    AUTHOR = {Sondow, J.},
     TITLE = {Zeros of the alternating zeta function on the line {$\mathfrak{R}(s)=1$}},
   JOURNAL = {Amer. Math. Monthly},
  FJOURNAL = {American Mathematical Monthly},
    VOLUME = {110},
      YEAR = {2003},
    NUMBER = {5},
     PAGES = {435--437},
      ISSN = {0002-9890,1930-0972},
   _MRCLASS = {11M41},
  _MRNUMBER = {2040887},
       DOI = {10.2307/3647831},
       URL = {https://doi.org/10.2307/3647831},
}

@article{s:5,
 Author = {{Speiser}, A.},
 Title = {{Geometrisches zur Riemannschen Zetafunktion.}},
 FJournal = {{Mathematische Annalen}},
 Journal = {{Math. Ann.}},
 ISSN = {0025-5831; 1432-1807/e},
 Volume = {110},
 Pages = {514--521},
 Year = {1934},
 Publisher = {Springer, Berlin/Heidelberg},
 Language = {German},
 MSC2010 = {11M26 11M06},
 Zbl = {0010.16401}
}

@article {yildirim,
    AUTHOR = {Y{\i}ld{\i}r{\i}m, C.Y.},
     TITLE = {Zeros of {$\zeta''(s)$} \& {$\zeta'''(s)$} in
              {$\sigma<\frac 12$}},
   JOURNAL = {Turkish J. Math.},
  FJOURNAL = {Turkish Journal of Mathematics},
    VOLUME = {24},
      YEAR = {2000},
    NUMBER = {1},
     PAGES = {89--108},
      ISSN = {1300-0098},
   _MRCLASS = {11M26},
  _MRNUMBER = {1792247},
_MRREVIEWER = {Bruce C. Berndt},
}

@Article{bp,
 Author = {{Boseman}, A. and {Pauli}, S.},
 Title = {{On the zeros of \(\zeta(s)-c\)}},
 FJournal = {{Involve}},
 Journal = {{Involve}},
 ISSN = {1944-4176},
 Volume = {6},
 Number = {2},
 Pages = {137--146},
 Year = {2013},
 Publisher = {Mathematical Sciences Publishers (MSP), Berkeley, CA},
 DOI = {10.2140/involve.2013.6.137},
 MSC2010 = {11M26},
 Zbl = {1291.11117}
}

@incollection {fps-icms,
    AUTHOR = {Farr, R.E. and Pauli, S. and Saidak, F.},
     TITLE = {Evaluating fractional derivatives of the {R}iemann zeta
              function},
 BOOKTITLE = {Mathematical software---{ICMS} 2020},
    SERIES = {Lecture Notes in Comput. Sci.},
    VOLUME = {12097},
     PAGES = {94--101},
 PUBLISHER = {Springer, Cham},
      YEAR = {[2020] \copyright 2020},
   _MRCLASS = {11M06},
  _MRNUMBER = {4139477},
       DOI = {10.1007/978-3-030-52200-1\_9},
}

@book{landau,
  address = {Leipzig},
  author = {Landau, E.},
  comment = {Early discussion of the O-Notation, by Landau personally. (pp.59-63)},
  note = {2 volumes. Reprinted by Chelsea, New York, 1953},
  publisher = {Teubner},
  title = {Handbuch der {Lehre} von der {Verteilung} der {Primzahlen}},
  year = 1909
}

@article{berndt,
title = "The number of zeros for $\zeta^{(k)}(s)$",
author = "Berndt, {B.C.}",
year = "1970",
doi = "10.1112/jlms/2.Part_4.577",
volume = "2",
pages = "577--580",
journal = "Journal of the London Mathematical Society",
issn = "0024-6107",
publisher = "Oxford University Press",
}

@article{spira3,
author = {Spira, R.},
title = {{An inequality for the Riemann zeta function}},
volume = {32},
journal = {Duke Mathematical Journal},
number = {2},
publisher = {Duke University Press},
pages = {247 -- 250},
year = {1965},
doi = {10.1215/S0012-7094-65-03223-0},
URL = {https://doi.org/10.1215/S0012-7094-65-03223-0}
}

@article{Jensen,
    author = {Jensen, J.},
    title = {{\em Sur un nouvel et important th{\' e}or{\` e}me de la th{\' e}orie des fonctions}},
    journal = {Acta Math.},
    year = {1899},
    pages = {359--364}
}

@book{mengoli,
    author ={Mengoli, P.},
    title = {Novae quadraturae arithmeticae},
    publisher = {Jacopo Monti},
    address = {Bologna},
    year = {1650} 
}

\bibliographystyle{amsalpha}

\end{document}